\documentclass[oneside,reqno]{amsart}
\usepackage[utf8]{inputenc}
\usepackage[margin=1.5in]{geometry}
\usepackage[title]{appendix}
\usepackage{amsthm}
\usepackage{url,color}
\usepackage{hyperref}
\usepackage{amsmath, amsfonts}
\usepackage{mathtools}
\usepackage{mathalfa}
\usepackage{amssymb}
\usepackage{upgreek}
\usepackage{mathrsfs}
\usepackage{relsize}

\usepackage[makeroom]{cancel}
\usepackage[mathscr]{eucal}
\numberwithin{equation}{section}

\newcommand{\dell}{\partial}

\newcommand{\ass}{\quad\hbox{as }}

\newcommand{\gradperp}{\nabla^{\perp}}

\newcommand{\R}{\mathbb R}
\newcommand{\nn}{\nabla} 
\newcommand{\ve}{\varepsilon} 
  
\newcommand{\be}{\begin{equation}} 
	\newcommand{\ee}{\end{equation}}

\newcommand{\bigO}{\mathcal{O}}
\newcommand{\twopartdef}[4]
{
	\left\{
	\begin{array}{ll}
		#1 & \text{if } #2 \\
		#3 & \text{if } #4
	\end{array}
	\right.
}

\DeclareMathOperator{\grad}{\nabla}

\DeclarePairedDelimiter\floor{\lfloor}{\rfloor}

\newtheorem{definition}{Definition}[section]

\newtheorem{theorem}[definition]{Theorem}

\newtheorem{lemma}[definition]{Lemma}

\newtheorem{remark}[definition]{Remark}

\usepackage[colorinlistoftodos]{todonotes}

\def\bcr{\begin{color}{red}}
	\def\bcb{\begin{color}{blue}}
		\def\ec{\end{color}}

	\date{}
	
\begin{document}
		
		\title[Vortex-pair solutions for incompressible Euler equations]{Asymptotic properties of  vortex-pair solutions for incompressible Euler equations in $\R^2$}
		
		\author[J. D\'avila]{Juan D\'avila}
		
		\address{J. D\'avila.- Department of Mathematical Sciences, University of Bath, Bath, Ba2 7AY, UK.}
		
		\email{jddb22@bath.ac.uk}%    \thanks will become a 1st page footnote.
		
		\author[M. del Pino]{Manuel del Pino}
		
		\address{M. del Pino.- Department of Mathematical Sciences, University of Bath, Bath, Ba2 7AY, UK.}
		
		\email{mdp59@bath.ac.uk}%    \thanks will become a 1st page footnote.
		
		\author[M.Musso]{Monica Musso}
		
		\address{M. Musso.- Department of Mathematical Sciences, University of Bath, Bath, Ba2 7AY, UK. }
		
		\email{m.musso@bath.ac.uk}
		
		\author[S. Parmeshwar ]{Shrish Parmeshwar }
		
		\address{S. Parmeshwar.- Mathematics Institute,
			University of Warwick,
			Coventry CV4 7AL,
			UK }
		
		\email{Shrish.Parmeshwar@warwick.ac.uk}

		\thanks{}


		\maketitle

		\begin{abstract}
			A {\em vortex pair} solution of the incompressible $2d$ Euler equation in vorticity form 
			$$ \omega_t  +  \nn^\perp  \Psi\cdot \nn \omega = 0  , \quad \Psi = (-\Delta)^{-1} \omega, \quad \hbox{in } \R^2 \times (0,\infty)$$
			is a travelling wave solution of the form $\omega(x,t)  =   W(x_1-ct,x_2 )$
			where $W(x)$ is compactly supported and odd in $x_2$. We revisit the problem of constructing solutions which are highly concentrated around points $ (0, \pm q)$, more precisely with approximately radially symmetric, compactly supported bumps with radius $\ve$ and masses $\pm m $.
			%$$W(x_1,x_2) =  \frac m{\ve^2} \big[  U\left (\frac {x- q e_2}{\ve}\right) - U\left (\frac {x+ q e_2}{\ve}\right) \big] .$$
			Fine asymptotic expressions are obtained, and the smooth dependence on the parameters $q$ and $\ve$  for the solution and its propagation speed $c$ are established. These results improve constructions through variational methods in \cite{smets-vanschaftingen} and in \cite{caowei} for the case of a bounded domain. 
			
		\end{abstract}
		
		\section{Introduction}
		
		This paper investigates the Euler equations for an incompressible, inviscid fluid in $\mathbb{R}^2$. In vorticity-stream formulation  these equations take the form 
		%\begin{subequations}\label{2d-euler-vorticity-stream} 
		
		\be \label{2d-euler-vorticity-stream}  \left\{   
		\begin{aligned}
			\omega_{t}+\gradperp\Psi\cdot\nabla\omega&=0\ \ \ \ &&\text{in}\ \mathbb{R}^{2}\times (0,T),\\ \Psi= (-\Delta)^{-1}&\omega\ \  &&\text{in}\ \mathbb{R}^{2}\times (0,T), \\
			\omega(\cdot,0)= &\mathring{\omega}  &&\text{in}\ \mathbb{R}^{2}.
		\end{aligned} \right.  \ee
		The stream function $\Psi$ is determined as the inverse Laplacian applied to $\omega$, yielding:
		$$
		\Psi(x,t)= (-\Delta)^{-1}\omega (x,t)   = \frac 1{2\pi} \int_{\mathbb R^2}\log \frac 1{|x-y|} \omega (y,t)\, dy \ \  
		$$
		and the fluid velocity field  $v =  \nn^\perp \Psi \coloneqq ( \Psi_{x_2}, -\Psi_{x_1})$  follows the Biot-Savart Law:
		$$
		v(x,t) =  \frac 1{2\pi} \int_{\mathbb R^2} \frac {(y-x)^\perp} {|y-x|^2} \omega (y,t)\, dy \ .
		$$
		The Cauchy problem \eqref{2d-euler-vorticity-stream} is well-posed in the weak sense in $L^1(\R^2)\cap L^\infty (\R^2)$ as deduced from the classical theories by Wolibner \cite{Wolibner1933} and Yudovich \cite{Yudovich1963}. Solutions are regular if the initial condition $\mathring{\omega}$ is so. See for instance Chapter 8 in Majda-Bertozzi's book \cite{majdabertozzi}.
		
		\medskip
		Of special interest are solutions with vorticity {\em concentrated} around a finite number of points $\xi_1(t), \ldots, \xi_N(t)$ in the form 
		\be \label{ome}
		\omega_\ve (x,t) =  \sum_{j=1}^N   \frac{m_j} {\ve^2}  \Big [  W\left ( \frac {x-\xi_j (t) } \ve   \right ) + o(1) \Big ] ,
		\ee
		where $W(y)$ is a fixed radially symmetric profile with $\int_{\R^2} W(y) dy = 1$, $o(1) \to 0 $ as $\ve \to 0$ and $m_j \in \mathbb{R}$ are the masses of the vortices. We have that 
		\begin{equation}\label{vortices-to-deltas}
			\omega_\ve(x,t)  \rightharpoonup  \omega_S(x,t) = \sum_{j=1}^N   {m_j} \delta_0  (  {x-\xi_j (t) } ) \ass \ve\to 0 ,
		\end{equation}
		where $\delta_0(y)$ designates a Dirac mass at $0$.
		%Formally we get   \begin{equation}\label{stream-to-deltas}\Psi_\ve(x,t)  \rightharpoonup  \Psi_S(x,t) = (-\Delta)^{-1} \omega_S(x,t) =  \sum_{j=1}^N  \frac{m_j}{2\pi} \log \frac 1 {|x-\xi_j (t)|} \ass \ve\to 0 .\end{equation} 
		%Substituting expressions \eqref{vortices-to-deltas}-\eqref{stream-to-deltas} into \eqref{2d-euler-vorticity-stream}, formally we get 
		%\begin{align*}
		%\partial_t \omega_S + \nabla^\perp \Psi_S \nabla \omega_S &= \sum_{j=1}^N \, m_j \left[ -\dot \xi_j -\sum_{\ell \not= j} {m_\ell \over 2\pi} {(x-\xi_\ell (t) )^\perp \over |x-\xi_\ell (t)|^2}  \right] \cdot \nabla \delta (x-\xi_j (t))  = 0. 
		%\end{align*}
		%??? notacion $S$ arriba o abajo 
		%
		%Since the distribution $\nabla \delta (x-\xi_j)$ is only supported at $x=\xi_j(t)$, we  
		Formally the trajectories $\xi_j(t)$ satisfy the Kirchhoff-Routh  vortex system
		\begin{equation}\label{KR}
			\dot \xi_j (t) = -\sum_{\ell \not= j} \frac{m_\ell}{2\pi} \frac{(\xi_j(t) -\xi_\ell (t) )^\perp}{|\xi_j(t) -\xi_\ell (t)|^2}, \quad j=1, \ldots , N.
		\end{equation}
		The study of point vortex dynamics evolving due to \eqref{KR} is itself a rich topic with many deep works, see \cite{aref} and references within.
		
		\medskip
		Finding regular solutions of \eqref{2d-euler-vorticity-stream} of the form \eqref{ome} that approximate a superposition of point vortices satisfying \eqref{KR} in the sense of \eqref{vortices-to-deltas} is a version of the classical \textit{vortex desingularization problem}, which has been the subject of many works, see for example \cite{MP1993, DDMW2020, caowei, smets-vanschaftingen} and further references inside.
		
		\medskip
		A key tool in finding solutions to the desingularization problem is to build them from the class of {\em steady solutions} to \eqref{2d-euler-vorticity-stream}, namely solutions whose shape does not depend on time. Examples include radially symmetric decaying functions $W(y)$, which correspond to stationary solutions. Other examples are traveling and rotating waves. See for instance \cite{smets-vanschaftingen,caowei}.
		
		\medskip
		One can find such solutions by reducing to semilinear elliptic problems, and a common method of doing so is the stream function method, which starts from the observation that if a $\Psi(x)$ satisfies the equation
		\begin{align}
			\Delta\Psi+f\left(\Psi\right)=0,\label{stream-function-method-equation}
		\end{align}
		for $f$ a $C^{1}$ function, then the stream-function/vorticity pair $\left(\Psi,f\left(\Psi\right)\right)$ is a stationary solution to \eqref{2d-euler-vorticity-stream}. From here, the family of rescaled solutions $\left(\Psi(x/\varepsilon),\varepsilon^{-2}f\left(\Psi(x/\varepsilon)\right)\right)$ is an example of a desingularized vortex solution for a single stationary point vortex.
		
		\medskip
		We are most in interested travelling wave solutions which can be constructed in a similar way; if $\omega(x,t)=\omega_{0}(x-cte_{1})$ for some constant $c$, then the transport equation in \eqref{2d-euler-vorticity-stream} is reduced to 
		\begin{align*}
			\gradperp\left(\Psi-cx_{2}\right)\cdot\grad\omega_{0}=0
		\end{align*}
		in moving coordinate frame $x'=x-cte_{1}$. Thus if we can find a $C^{1}$ functions $f$ and a solution $\Psi$ to
		\begin{align*}
			\Delta\Psi+f\left(\Psi-cx_{2}\right)=0,
		\end{align*}
		we have a travelling wave solution to \eqref{2d-euler-vorticity-stream}.
		
		\medskip
		The aim of this paper is to construct and give certain fine information for a desingularization of the simplest travelling wave solution to the Kirkhhoff-Routh system \eqref{KR}, a vortex pair. Let $q>0$, $M>0$. Then a solution to \eqref{KR} for $N=2$ is given by
		\begin{align}
			\xi_{1}(t)=qe_{2}+cte_{1},\quad \xi_{2}(t)=-qe_{2}+cte_{1},\quad c=\frac{M}{4\pi q},\label{point-vortex-pair-solution}
		\end{align}
		which clearly describes two point vortices a distance $2q$ apart in the $e_{2}$ direction with opposite masses travelling uniformly in the $e_{1}$ direction at speed $c$.
		
		\medskip
		Desingularizing vortex pairs has been a topic of great interest, yielding numerous results in various contexts. Our paper focuses on constructing a desingularized vortex pair solution to \eqref{2d-euler-vorticity-stream}. This solution has previously been established by Smets and Van Schaftingen, in \cite{smets-vanschaftingen}. In their work, they treated the problem using variational techniques, obtaining least energy solutions by minimizing the associated energy functional.
		Cao, Liu, and Wei, in \cite{caowei}, addressed the case of bounded domains through a reduction method. Further results in this direction can be found in \cite{cao,norbury}, and earlier constructions are explored in \cite{burton,turkington} by Burton and Turkington.
		
		\medskip
		Vortex pairs have also been constructed and studied for related systems, see \cite{ao} for a vortex pair solution to the surface quasi-geostrophic equation, and \cite{LiuWei,ChiPach} for travelling wave solutions to the Gross-Pitaevskii equation.
		
		\medskip
		Contrasting with \cite{smets-vanschaftingen}, in this paper we construct the solution instead using Lyapunov-Schmidt reduction, give precise estimates in terms of $\varepsilon$ on lower order terms, and show differentiability of the solution with respect to the parameter $q$.
		
		\medskip
		We now describe the desingularized vortex pair $\left(\Psi_{\varepsilon},\omega_{\varepsilon}\right)$ that solves \eqref{2d-euler-vorticity-stream}. The vorticity takes the form
		\begin{align}
			\omega_{\varepsilon}(x,t)=U_{q,\varepsilon}\left(x-cte_{1}\right),\nonumber
		\end{align}
		where $ U_{q,\varepsilon} = \varepsilon^{-2} f_\varepsilon( \Psi _{\varepsilon}(x) - c_{q,\varepsilon} x_2  ) $ for an appropriate nonlinearity $f_\varepsilon$, with $\Psi_{\varepsilon}$ solving the elliptic problem
		\begin{align}
			-\Delta \Psi_{\varepsilon} (x) =  \varepsilon^{-2} f_\varepsilon( \Psi_{\varepsilon} (x) - c_{q,\varepsilon} x_2  ) .
			\label{stationary-solutions-elliptic-pde} 
		\end{align}
		in the upper half plane $x_2>0$, and is even in $x_1$ and odd in $x_2$. Our stream function $\Psi $ has the form
		\begin{equation} \label{travelling-vortex-pair-ansatz}
			\begin{aligned}
				\Psi_{\varepsilon}(x)&=\mathring{\Psi}_{\varepsilon}(x)+\psi_{\frac{q}{\varepsilon}}\left(\frac{x}{\varepsilon}\right), \quad
				\mathring{\Psi}_{\varepsilon}(x)=\Gamma\left(\frac{x-q e_{2}}{\varepsilon}\right)-\Gamma\left(\frac{x+q e_{2}}{\varepsilon}\right),
			\end{aligned}
		\end{equation}
		where  $\Gamma(x)$ is an entire radial classical solution of the problem
		\begin{align}
			\Delta \Gamma+\Gamma^{\gamma}_{+}=0 \ \text{on}\ \mathbb{R}^{2} , \quad \{ \Gamma>0\} = B_1,
			\label{power-semilinear-problem-R2}
		\end{align}
		$\gamma>3$, $\psi_{q/\varepsilon}$ is of order $\varepsilon^{2}$, and is even in $x_{1}$ and odd in $x_{2}$ so that $\Psi$ exhibits the symmetries we desire. The speed $c_{q,\varepsilon} $ has the expansion
		\begin{align}
			c_{q,\varepsilon}  = \frac M{4\pi q} + O(\ve^2),\quad M\coloneqq\int_{\mathbb{R}^{2}}\Gamma^{\gamma}_{+}dx.\label{vortex-mass-definition}
		\end{align}
		Suppressing the dependency of $c_{q,\varepsilon}$ on $(q,\varepsilon)$ and instead writing $c$, the nonlinearity $f_{\varepsilon}$ in \eqref{stationary-solutions-elliptic-pde} is defined by
		\begin{align}
			f_{\varepsilon}(\Psi(x)-cx_{2})&=\left(\Psi_{\varepsilon}(x)-cx_{2}-\left|\log{\varepsilon}\right|\Omega\right)^{\gamma}_{+}\chi_{B_{s}\left(q e_{2}\right)}(x)\nonumber\\
			&-\left(-\Psi_{\varepsilon}(x)+cx_{2}-\left|\log{\varepsilon}\right|\Omega\right)^{\gamma}_{+}\chi_{B_{s}\left(-q e_{2}\right)}(x).\label{nonlinearity-definition}
		\end{align}
		Here $s>0$ is a fixed constant, and the $\chi$ are indicator functions. To find such a solution requires choosing our $c$ and $\Omega$ as a function of $q$. The quantity $\Omega$ is a parameter to adjust the mass of our vortex pair. 
		We demand that
		\begin{align}
			-\Gamma\left(\frac{2q e_{2}}{\varepsilon}\right)-cq-\left|\log{\varepsilon}\right|\Omega=0.\label{Omega-definition-1}
		\end{align}
		This choice guarantees that
		$$
		M - \varepsilon^{-2}\int_{\mathbb{R}^{2}}\left(\mathring{\Psi}_{\varepsilon}(x)-cx_{2}-\left|\log{\varepsilon}\right|\Omega\right)^{\gamma}_{+}\chi_{B_{s}\left(q e_{2}\right)}(x)\ dx =O(\ve ), \ass \varepsilon \to 0.
		$$
		If we  suppose  $q\in\left(Q_{0}^{-1},Q_{0}\right)$, some $Q_{0}>0$ large and   $c=c(q,\varepsilon)$ has the bound $c\leq mQ_{0}$ for all $\varepsilon$ small enough, then 
		the support of $f_{\varepsilon}(\mathring{\Psi}_{\varepsilon}(x)-cx_{2})$ shrinks as $\varepsilon\to0$:   there exists $\varepsilon_{0}$ small enough such that for all $\varepsilon\leq\varepsilon_{0}$,  
		\begin{align}
			{\mbox {supp}} \, f_{\varepsilon}\left(\mathring{\Psi}_{\varepsilon}(x)-cx_{2}\right) \subseteq B_{\rho_0\varepsilon}(q e_{2})\cup B_{\rho_0\varepsilon}(-q e_{2})\label{rho-0-support}
		\end{align}
		for
		\begin{align}
			\rho_0=2\left(1+\frac{s Q_{0}}{2}\right)e^{Q_{0}^{2}},\label{rho-0-definition}
		\end{align}
		and
		\begin{align}
			\left|\left(\Gamma(y-\frac{q}{\varepsilon} e_{2})\right)^{\gamma}_{+}-\left(\mathring{\Psi}_{\varepsilon}(\varepsilon y)-c\varepsilon y_{2}-\left|\log{\varepsilon}\right|\Omega\right)^{\gamma}_{+}\chi_{B_{\frac{s}{\varepsilon}}\left(\frac{q}{\varepsilon} e_{2}\right)}\right|\leq C\varepsilon\chi_{B_{\rho_0}(\frac{q}{\varepsilon}e_{2})}(y),\label{error-estimate}
		\end{align}
		with an analogous statement holding around $-\frac{q}{\varepsilon} e_{2}$ via symmetry in $y_{2}$.
		
		\medskip
		One can write down the form of $\Gamma$ more explicitly. We have
		\begin{align}
			\Gamma=\twopartdef{\nu(|x|)}{|x|\leq1}{\nu'(1)\log{|x|}}{|x|>1},\label{w-definition}
		\end{align}
		for $\nu$ the unique radial ground state of the problem
		\begin{align}
			\Delta \nu+\nu^{\gamma}=0,\ \nu\in H^{1}_{0}(B_{1}),\ \nu>0\ \text{on}\ B_{1}.\label{basic-power-semilinear-problem}
		\end{align}
		A quantity we refer to repeatedly in our construction is something we call the reduced mass, $M/2\pi$, where $M$ is defined in \eqref{vortex-mass-definition}. From \eqref{w-definition} and the fact that $\nu$ solves \eqref{basic-power-semilinear-problem}, we can see
		\begin{align}
			M=2\pi\left|\nu'(1)\right|, \quad  m\coloneqq\frac{M}{2\pi}=\left|\nu'(1)\right|.\nonumber
		\end{align}
		Rescalings of $\Gamma$ are clearly $C^{1}$ across the boundary of the ball radius $\varepsilon$ and solve
		\begin{align}
			\Delta\Gamma\left(\frac{x}{\varepsilon}\right)+\varepsilon^{-2}\left(\Gamma\left(\frac{x}{\varepsilon}\right)\right)^{\gamma}_{+}=0,\nonumber
		\end{align}
		which is of the form of a desingularized vortex obtained by the stream function method described in \eqref{stream-function-method-equation}. Thus the form of the nonlinearity \eqref{nonlinearity-definition} as well as the choice of $\Omega$ \eqref{Omega-definition-1} also mean that on $B_{\rho_0\varepsilon}(q e_{2})\cup B_{\rho_0\varepsilon}(-q e_{2})$, our vortex pair is a good approximation of these steady state vortices in the moving frame.
		
		\medskip
		Our nonlinearity $f_\varepsilon$ in \eqref{nonlinearity-definition} has the desired regularity due to the fact that our choice of cutoff represented by the indicator functions means that its support will actually be contained in $\chi_{B_{s'}\left(q e_{2}\right)}(x)\cup\chi_{B_{s'}\left(-q e_{2}\right)}(x)$, some $s'<s$, for small enough $\varepsilon$, meaning the nonlinearity is $C^{\floor{\gamma},\gamma-\left[\gamma\right]}$  when $\gamma\notin \mathbb{Z}$ and $C^{\gamma-1}$ otherwise.
		
		\medskip
		Our main theorem can now be stated.
		\begin{theorem}\label{vortexpair-theorem} Let $\gamma > 3$ in \eqref{power-semilinear-problem-R2} and let $q >0.$ For all $\varepsilon >0$ small, there exist  constants $C>0$, $c$ and a function $\psi_{\frac{q}{\varepsilon}} (x)$ even in $x_1$, odd in $x_2$ such that $\Psi$ given by \eqref{travelling-vortex-pair-ansatz} with $\Omega$ as in \eqref{Omega-definition-1} solves  \eqref{stationary-solutions-elliptic-pde}, with $f_\varepsilon$ given by \eqref{nonlinearity-definition}. Set
			$$
			q'=\frac{q}{\varepsilon}.
			$$
			The speed $c$ is given by
			\begin{equation}\nonumber
				c= \frac{m}{2q} + g(q' )\end{equation}
			for a $C^{2}$ function $g$, and the following estimates hold true
			\begin{align}
				\| \psi_{q'} \|_{L^\infty (\mathbb{R}^2)} &+  \| \partial_{q'} \psi_{q'} \|_{L^\infty (\mathbb{R}^2)}+\| \partial_{q'}^{2} \psi_{q'} \|_{L^\infty (\mathbb{R}^2)} \leq C \ve^2,\label{psi-q-l-infinity-bounds} \\
				\varepsilon^{2}|g| &+ \varepsilon|\partial_{q'} g|+|\partial_{q'}^{2} g| \leq C \ve^4.\label{g-pointwise-bounds}
			\end{align}
		\end{theorem}
		In addition to the properties stated in Theorem \ref{vortexpair-theorem}, we can also describe to main order the expansion in Fourier modes on $B_{\rho\varepsilon}(q e_{2})\cup B_{\rho\varepsilon}(-q e_{2})$, some constant $\rho>0$ independent of $\varepsilon$, see Remark \ref{mode-remark}.
		
		\medskip
		The structure of the paper is as follows. In Section \ref{vortex-pair-appendix}, we construct the vortex pair. This is done by solving a nonlinear problem for $\psi_{\frac{q}{\varepsilon}}$, using Lyapunov-Schmidt reduction, in a similar strategy to the one laid out in \cite{ao}. First, we will solve a related, projected linear problem. We will then use this linear theory we develop to solve a projected nonlinear problem by a fixed point argument. From there, our solution will be found by solving an algebraic equation involving $c$ and $q$ that will imply the solution to the projected nonlinear problem is actually our desired solution to the full problem. Where there is similarity to the techniques shown in \cite{ao}, some details will be omitted.
		
		\medskip
		In Section \ref{q-linearisation-properties-section}, we will show the existence of the derivatives of $\psi_{\frac{q}{\varepsilon}}$ and $c$ with respect to $q$, as well as bounds \eqref{psi-q-l-infinity-bounds}--\eqref{g-pointwise-bounds}. This is a delicate issue as will see in Section \ref{vortex-pair-appendix} that via the Lyapunov-Schmidt reduction, $\psi_{\frac{q}{\varepsilon}}$ and $c$ solve a coupled system of equations. The key observation is that the semilinear elliptic equation for the difference quotient of $\psi_{\frac{q}{\varepsilon}}$ with respect to $q$ comes with a term on the right hand side that is multiplied by the difference quotient of $c$ with respect to $q$. The structure of this term on the right hand side lets us adapt the linear theory we develop in Section \ref{vortex-pair-appendix} to simultaneously get good bounds on these two difference quotients, allowing us to show existence of the $q$ derivatives, and finally  \eqref{psi-q-l-infinity-bounds}--\eqref{g-pointwise-bounds}.
		
		\section{The Vortex Pair}\label{vortex-pair-appendix}
		\subsection{Nonlinear Problem for the Remainder}
		From now on, throughout the paper, $\gamma>3$, and $q=\varepsilon q'$. We construct a travelling vortex pair solution, for all $\varepsilon>0$ small enough, solving
		\begin{align}
			\Delta \Psi_{\varepsilon}+\varepsilon^{-2}f_{\varepsilon}\left(\Psi_{\varepsilon}-cx_{2}\right)=0\label{vortex-pair-semilinear-equation-appendix}
		\end{align}
		of the form \eqref{travelling-vortex-pair-ansatz} with $\Gamma$ defined in \eqref{w-definition}, and $f_{\varepsilon}$, defined in \eqref{nonlinearity-definition}. Note that $\mathring{\Psi}_{\varepsilon}$ also defined in \eqref{travelling-vortex-pair-ansatz} is odd in $x_{2}$, and even in $x_{1}$. We will find a solution to \eqref{vortex-pair-semilinear-equation-appendix} with $\psi_{q'}$ having the same symmetries as $\mathring{\Psi}_{\varepsilon}$. That is
		\begin{align}
			\psi_{q'}(-x_{1},x_{2})=\psi_{q'}(x_{1},x_{2})=-\psi_{q'}(x_{1},-x_{2}).\label{solution-symmetry-assumptions}
		\end{align}
		Let us first rewrite our problem as a nonlinear problem in $\psi_{q'}$. Substituting \eqref{travelling-vortex-pair-ansatz} in to the left hand side of \eqref{vortex-pair-semilinear-equation-appendix}, we get a solution to our problem of the form \eqref{travelling-vortex-pair-ansatz} if and only if 
		\begin{align}
			S[\psi_{\varepsilon}]&\coloneqq\Delta\left(\Gamma\left(\frac{x-qe_{2}}{\varepsilon}\right)-\Gamma\left(\frac{x+qe_{2}}{\varepsilon}\right)+\psi_{q'}\left(\frac{x}{\varepsilon}\right)\right)\nonumber\\
			&+\varepsilon^{-2}\left(\Gamma\left(\frac{x-qe_{2}}{\varepsilon}\right)-\Gamma\left(\frac{x+qe_{2}}{\varepsilon}\right)+\psi_{q'}\left(\frac{x}{\varepsilon}\right)-cx_{2}-\left|\log{\varepsilon}\right|\Omega\right)^{\gamma}_{+}\chi_{B_{s}\left(qe_{2}\right)}\nonumber\\
			&-\varepsilon^{-2}\left(\Gamma\left(\frac{x+qe_{2}}{\varepsilon}\right)-\Gamma\left(\frac{x-qe_{2}}{\varepsilon}\right)-\psi_{q'}\left(\frac{x}{\varepsilon}\right)+cx_{2}-\left|\log{\varepsilon}\right|\Omega\right)^{\gamma}_{+}\chi_{B_{s}\left(-qe_{2}\right)}\nonumber\\
			&=0.\label{psi-q-fixed-point-problem-formulation-1}
		\end{align}
		Then, letting $x=\varepsilon y$, condition \eqref{psi-q-fixed-point-problem-formulation-1} is equivalent to
		\begin{align}
			\Tilde{S}[\psi_{\varepsilon}]&\coloneqq\Delta\left(\Gamma\left(y-q'e_{2}\right)-\Gamma\left(y+q'e_{2}\right)+\psi_{q'}\left(y\right)\right)\nonumber\\
			&+\left(\Gamma\left(y-q'e_{2}\right)-\Gamma\left(y+q'e_{2}\right)+\psi_{q'}\left(y\right)-c\varepsilon y_{2}-\left|\log{\varepsilon}\right|\Omega\right)^{\gamma}_{+}\chi_{B_{\frac{s}{\varepsilon}}\left(q'e_{2}\right)}\nonumber\\
			&-\left(\Gamma\left(y+q'e_{2}\right)-\Gamma\left(y-q'e_{2}\right)-\psi_{q'}\left(y\right)+c\varepsilon y_{2}-\left|\log{\varepsilon}\right|\Omega\right)^{\gamma}_{+}\chi_{B_{\frac{s}{\varepsilon}}\left(-q'e_{2}\right)}\nonumber\\
			&=0.\nonumber
		\end{align}
		We wish to write the above as something amenable to a fixed point problem in $\psi_{q'}$, and so we will write the nonlinearity as a linearisation around $\Gamma\left(y-q'e_{2}\right)-\Gamma\left(y+q'e_{2}\right)$ plus error terms. To do this we first define some quantities. First we define our potential:
		\begin{align}
			V(y)&=\gamma\left(\Gamma\left(y-q'e_{2}\right)-\Gamma\left(y+q'e_{2}\right)-c\varepsilon y_{2}-\left|\log{\varepsilon}\right|\Omega\right)^{\gamma-1}_{+}\chi_{B_{\frac{s}{\varepsilon}}\left(q'e_{2}\right)}\nonumber\\
			&+\gamma\left(\Gamma\left(y+q'e_{2}\right)-\Gamma\left(y-q'e_{2}\right)+c\varepsilon y_{2}-\left|\log{\varepsilon}\right|\Omega\right)^{\gamma-1}_{+}\chi_{B_{\frac{s}{\varepsilon}}\left(-q'e_{2}\right)},\nonumber\\
			&\coloneqq V^{+}+V^{-}.\nonumber
		\end{align}
		Next we define our error terms. Let
		\begin{align}
			E_{1}(y)&=\left(\Gamma(y-q'e_{2})\right)^{\gamma}_{+}-\left(\Gamma\left(y-q'e_{2}\right)-\Gamma\left(y+q'e_{2}\right)-c\varepsilon y_{2}-\left|\log{\varepsilon}\right|\Omega\right)^{\gamma}_{+}\chi_{B_{\frac{s}{\varepsilon}}\left(q'e_{2}\right)},\label{error-1-definition}\\
			E_{2}(y)&=\left(\Gamma(y+q'e_{2})\right)^{\gamma}_{+}-\left(\Gamma\left(y+q'e_{2}\right)-\Gamma\left(y-q'e_{2}\right)+c\varepsilon y_{2}-\left|\log{\varepsilon}\right|\Omega\right)^{\gamma}_{+}\chi_{B_{\frac{s}{\varepsilon}}\left(-q'e_{2}\right)}\label{error-2-defintion}
		\end{align}
		Finally we define our nonlinearities, depending on $\psi_{q'}$:
		\begin{align}
			N_{1}[\psi_{q'}]&=\left(\Gamma\left(y-q'e_{2}\right)-\Gamma\left(y+q'e_{2}\right)+\psi_{q'}\left(y\right)-c\varepsilon y_{2}-\left|\log{\varepsilon}\right|\Omega\right)^{\gamma}_{+}\chi_{B_{\frac{s}{\varepsilon}}\left(q'e_{2}\right)}\nonumber\\
			&-\left(\Gamma\left(y-q'e_{2}\right)-\Gamma\left(y+q'e_{2}\right)-c\varepsilon y_{2}-\left|\log{\varepsilon}\right|\Omega\right)^{\gamma}_{+}\chi_{B_{\frac{s}{\varepsilon}}\left(q'e_{2}\right)}\nonumber\\
			&-\gamma\left(\left(\Gamma\left(y-q'e_{2}\right)-\Gamma\left(y+q'e_{2}\right)-c\varepsilon y_{2}-\left|\log{\varepsilon}\right|\Omega\right)^{\gamma-1}_{+}\psi_{q'}(y)\right)\chi_{B_{\frac{s}{\varepsilon}}\left(q'e_{2}\right)}\label{nonlinearity-1-definition}\\
			N_{2}[\psi_{q'}]&=\left(\Gamma\left(y+q'e_{2}\right)-\Gamma\left(y-q'e_{2}\right)-\psi_{q'}\left(y\right)+c\varepsilon y_{2}-\left|\log{\varepsilon}\right|\Omega\right)^{\gamma}_{+}\chi_{B_{\frac{s}{\varepsilon}}\left(-q'e_{2}\right)}\nonumber\\
			&-\left(\Gamma\left(y+q'e_{2}\right)-\Gamma\left(y-q'e_{2}\right)+c\varepsilon y_{2}-\left|\log{\varepsilon}\right|\Omega\right)^{\gamma}_{+}\chi_{B_{\frac{s}{\varepsilon}}\left(-q'e_{2}\right)}\nonumber\\
			&+\gamma\left(\left(\Gamma\left(y+q'e_{2}\right)-\Gamma\left(y-q'e_{2}\right)+c\varepsilon y_{2}-\left|\log{\varepsilon}\right|\Omega\right)^{\gamma-1}_{+}\psi_{q'}(y)\right)\chi_{B_{\frac{s}{\varepsilon}}\left(-q'e_{2}\right)}.\label{nonlinearity-2-definition}
		\end{align}
		Then letting
		\begin{align}
			E(y)=E_{1}(y)-E_{2}(y),\label{full-error-definition}\\
			N[\psi_{q'}]=N_{1}[\psi_{q'}]-N_{2}[\psi_{q'}],\label{full-nonlinearity-definition}
		\end{align}
		the condition $\Tilde{S}[\psi_{\varepsilon}]=0$ is equivalent to
		\begin{align}
			L[\psi_{q'}]\coloneqq\Delta\psi_{q'}+V\psi_{q'}=E-N[\psi_{q'}],\label{psi-q-fixed-point-problem-formulation-3}
		\end{align}
		where we have used the fact that
		\begin{align*}
			\Delta\left(\Gamma\left(y-q'e_{2}\right)-\Gamma\left(y+q'e_{2}\right)\right)=-\left(\left(\Gamma\left(y-q'e_{2}\right)\right)^{\gamma}_{+}-\left(\Gamma\left(y+q'e_{2}\right)\right)^{\gamma}_{+}\right).
		\end{align*}
		To finish the first part of this section, we will precisely state the problem we are looking to solve, including the function space we will find a solution in. To do this, we define
		\begin{equation}\label{dq-approximate-kernel-element}
			\begin{aligned}
				Z_{1}(y)=\dell_{1}\Gamma\left(y-q'e_{2}\right)-\dell_{1}\Gamma\left(y+q'e_{2}\right)=Z_{1,1}(y)-Z_{1,2}(y),\\
				Z_{2}(y)=\dell_{2}\Gamma\left(y-q'e_{2}\right)+\dell_{2}\Gamma\left(y+q'e_{2}\right)=Z_{2,1}(y)+Z_{2,2}(y).
			\end{aligned}
		\end{equation}
		We first want to solve a linear projected version of \eqref{psi-q-fixed-point-problem-formulation-3}, where the projection is with respect to $Z_{1}$ and $Z_{2}$.
		
		\medskip
		We notice that both $Z_{1}$ and $Z_{2}$ are odd in $y_{2}$, with $Z_{1}$ odd in $y_{1}$, and $Z_{2}$ even in $y_{1}$. Therefore any solution $\psi$ satisfying the symmetry conditions \eqref{solution-symmetry-assumptions} will automatically have
		\begin{align}
			\int_{\mathbb{R}^{2}}VZ_{1}\psi=\int_{\mathbb{R}^{2}}VZ_{1}Z_{2}=0.\nonumber
		\end{align}
		Thus we need only to project with respect to $Z_{2}$. To that end, we will also demand our solutions satisfy
		\begin{align}
			\int_{\mathbb{R}^{2}}VZ_{2}\psi=0.\label{solution-Z2-orthogonality-condition}
		\end{align}
		Now, let $X$ be the Banach space defined as follows
		\begin{align}
			X\coloneqq\left\{f\in L^{\infty}\left(\mathbb{R}^{2}\right)|f\ \text{satisfies}\ \eqref{solution-symmetry-assumptions}\ \text{and}\ \eqref{solution-Z2-orthogonality-condition}\right\},\label{solution-banach-space}
		\end{align}
		equipped with the standard $L^{\infty}\left(\mathbb{R}^{2}\right)$ norm. Then the problem we are trying to solve can be precisely stated as
		\begin{align}
			L[\psi_{q'}]&=E-N[\psi_{q'}]\ \text{on}\ \mathbb{R}^{2},\nonumber\\
			\psi_{q'}&\in X.\nonumber
		\end{align}
		\subsection{Projected Linear Problem}\label{projected-linear-problem-subsection}
		It will be useful to construct a slightly more general linear theory than is required given the discussed in the previous section. Thus we consider the function space consisting of $\psi\in L^{\infty}\left(\mathbb{R}^{2}\right)$ that satisfy
		\begin{align}
			\psi(y_{1},-y_{2})=-\psi(y_{1},y_{2}),\label{generalised-symmetry-assumptions}
		\end{align}
		satisfying the orthogonality conditions
		\begin{align}
			\int_{\mathbb{R}^{2}}VZ_{i}\psi=0,\ i=1,2.\label{generalised-orthogonality-conditions}
		\end{align}
		To that end, define $Y$ to be the function space
		\begin{align}
			Y\coloneqq\left\{f\in L^{\infty}\left(\mathbb{R}^{2}\right)|f\ \text{satisfies}\ \eqref{generalised-symmetry-assumptions}\ \text{and}\ \eqref{generalised-orthogonality-conditions}\right\}\label{generalised-banach-space}
		\end{align}
		equipped with the standard $L^{\infty}\left(\mathbb{R}^{2}\right)$ norm. Note that $X\subset Y$, with $X$ defined in \eqref{solution-banach-space}.
		
		\medskip
		With these definitions in hand, we first consider the projected linear problem
		\begin{align}
			L[\psi]&=h+aVZ_{1}+bVZ_{2}\ \text{on}\ \mathbb{R}^{2},\nonumber\\
			\psi&\in Y.\label{projected-linear-problem-formulation-1}
		\end{align}
		Here $a,b\in\mathbb{R}$, and $h$ is a function satisfying
		\begin{align}
			\left\|h\right\|_{**}\coloneqq\sup_{y\in\mathbb{R}^{2}}W(y)^{-\left(2+\sigma\right)}\left|h(y)\right|<\infty,\nonumber
		\end{align}
		with $\sigma\in(0,1)$, and
		\begin{align}
			W(y)=\frac{1}{1+\left|y-q'e_{2}\right|}+\frac{1}{1+\left|y+q'e_{2}\right|}.\label{sigma-weight-definition}
		\end{align}
		
		We first have a priori estimates on solutions to \eqref{projected-linear-problem-formulation-1}.
		
		\begin{lemma}\label{estimate-on-coefficient-of-projection-lemma}
			Let $(q,c,\varepsilon)$ be chosen so that \eqref{rho-0-support}--\eqref{rho-0-definition} hold. Suppose we have a solution to \eqref{projected-linear-problem-formulation-1}, with $h$ satisfying $\left\|h\right\|_{**}<\infty$ and the symmetry assumption \eqref{generalised-symmetry-assumptions}. Then
			\begin{align}
				\left(\left|a\right|+\left|b\right|\right)\leq C\left(\varepsilon\left\|\psi\right\|_{L^{\infty}\left(\mathbb{R}^{2}\right)}+\left\|h\right\|_{**}\right),\nonumber
			\end{align}
			some constant $C>0$.
		\end{lemma}
		\begin{proof}
			See \cite{ao}.
		\end{proof}
		Next is an a priori estimate for the solution of \eqref{projected-linear-problem-formulation-1} in terms of $\left\|h\right\|_{**}$.
		\begin{lemma}\label{psi-a-priori-estimate-lemma}
			Let $(q,c,\varepsilon)$ be chosen so that \eqref{rho-0-support}--\eqref{rho-0-definition} hold.. Suppose we have a solution to \eqref{projected-linear-problem-formulation-1}, with $h$ satisfying $\left\|h\right\|_{**}<\infty$ and the symmetry assumption \eqref{generalised-symmetry-assumptions}. Then
			\begin{align}
				\left\|\psi\right\|_{L^{\infty}\left(\mathbb{R}^{2}\right)}\leq C \left\|h\right\|_{**},\nonumber
			\end{align}
			some constant $C>0$.
		\end{lemma}
		\begin{proof}
			This argument is again very similar to the corresponding a priori estimate in \cite{ao}. We proceed by contradiction and suppose there is a sequence of $\varepsilon_{n}$ such that $\varepsilon_{n}\to0$ as $n\to\infty$. Suppose further that for this sequence, we have a sequence of solutions $\psi_{n}$ to \eqref{projected-linear-problem-formulation-1} such that
			\begin{align}
				\left\|\psi_{n}\right\|_{L^{\infty}\left(\mathbb{R}^{2}\right)}&=1\ \text{for all}\ n,\label{psi-n-all-norm-1}\\
				\left\|h_{n}\right\|_{**}&\to0\ \text{as}\ n\to\infty.\label{log-R-delta-norm-h-goes-to-0}
			\end{align}
			Let $q_{n}=\varepsilon_{n}q'_{n}$ for some sequence $q_{n}\in(Q_{0}^{-1},Q_{0})$ for all $n$. Then as in \cite{ao}, \eqref{psi-n-all-norm-1}--\eqref{log-R-delta-norm-h-goes-to-0} implies
			\begin{align}
				\left(\left\|\psi_{n}\right\|_{L^{\infty}\left(B_{R}(q'_{n}e_{2})\right)}+\left\|\psi_{n}\right\|_{L^{\infty}\left(B_{R}(-q'_{n}e_{2})\right)}\right)\to0\ \text{as}\ n\to\infty,\label{psi-n-do-not-concentrate}
			\end{align}
			for any fixed radius $R>0$, due to the orthogonality conditions \eqref{generalised-orthogonality-conditions} and the non-degeneracy argument for the linearized operator around the power nonlinearity vortex \eqref{power-semilinear-problem-R2} given in \cite{DancerYan}. In particular it is true for $R=2\rho$, where the support of $V_{n}$ in contained in $B_{\rho}\left(q'_{n}e_{2}\right)$ for all $n$.
			
			\medskip
			Next, we obtain bounds for $\psi_{n}(y)$ on the region $|y-q_{n}'e_{2}|\geq3q_{n}'$. Using the Green's function to write,
			\begin{align}
				\psi_{n}(y)&=\frac{1}{2\pi}\int_{\mathbb{R}_{+}^{2}}\log{\left(\frac{\left|z-\bar{y}\right|}{\left|z-y\right|}\right)}\left(h_{n}(z)-V_{n}\psi_{n}+a_{n}V_{n}Z_{1}^{(n)}+b_{n}V_{n}Z_{2}^{(n)}\right)\ dz,\nonumber
			\end{align}
			we have 
			\begin{align}
				\left\|\psi_{n}\right\|_{L^{\infty}\left(\{|y-q_{n}'e_{2}|\geq3q_{n}'\}\right)}\leq C\left(\left\|h_{n}\right\|_{**}+\left\|\psi_{n}\right\|_{L^{\infty}\left(B_{\rho}(q'_{n}e_{2})\right)}+\left|a_{n}\right|+\left|b_{n}\right|\right).\label{psi-n-goes-to-0-15}
			\end{align}
			Now let $G_{n}$ be the annulus defined by $|y-q_{n}'e_{2}|\in[2\rho,3q_{n}']$. We have shown that
			\begin{align}
				\left\|\psi_{n}\right\|_{L^{\infty}\left(\dell G_{n}\right)}\leq CK_{n},\label{psi-n-K-n-boundary-bound}
			\end{align}
			where 
			\begin{align}
				K_{n}=\max\left\{\left\|\psi_{n}\right\|_{L^{\infty}\left(\dell B_{2\rho}(q'_{n}e_{2})\right)},\left\|h_{n}\right\|_{**}+\left\|\psi_{n}\right\|_{L^{\infty}\left(B_{\rho}(q'_{n}e_{2})\right)}+\left|a_{n}\right|+\left|b_{n}\right|\right\}.\label{Kn-definition}
			\end{align}
			On the annulus $G_{n}$,
			\begin{align}
				-\Delta\psi_{n}=h_{n}.\label{psi-n-h-n-annulus-relation}
			\end{align}
			Define $\beta_{n}$ by
			\begin{align}
				\beta_{n}(y)=c_{1}K_{n}\left(1-\left|y-q_{n}'e_{2}\right|^{-\sigma}\right),\nonumber
			\end{align}
			for $c_{1}$ chosen appropriately so that
			\begin{align}
				\beta_{n}(y)\geq\psi_{n}(y)\nonumber
			\end{align}
			on $\dell G_{n}$. Such a choice is possible on this region as
			\begin{align}
				1-\left(2\rho\right)^{-\sigma}\leq1-\left|y-q_{n}'e_{2}\right|^{-\sigma}\leq1-\left(3q_{n}'\right)^{-\sigma}.\nonumber
			\end{align}
			Note that $-\Delta\beta_{n}= c_{1}\sigma^{2}\left|y-q_{n}'e_{2}\right|^{-(2+\sigma)}K_{n}$. We know from \eqref{Kn-definition} that
			\begin{align*}
				K_{n}\geq \left\|h_{n}\right\|_{**}+\left\|\psi_{n}\right\|_{L^{\infty}\left(B_{\rho}(q'_{n}e_{2})\right)}+\left|a_{n}\right|+\left|b_{n}\right|\geq \left\|h_{n}\right\|_{**}.
			\end{align*}
			Therefore
			\begin{align*}
				&\left|y-q_{n}'e_{2}\right|^{-(2+\sigma)}K_{n}\geq \left|y-q_{n}'e_{2}\right|^{-(2+\sigma)}\left\|h_{n}\right\|_{**}\\
				&=\left|y-q_{n}'e_{2}\right|^{-(2+\sigma)}\sup_{z\in\mathbb{R}^{2}_{+}}W(z)^{-\left(2+\sigma\right)}\left|h_{n}(z)\right|,
			\end{align*}
			where we can take the supremum over $z\in\mathbb{R}^{2}_{+}$ due to the symmetry assumption \eqref{generalised-symmetry-assumptions}, and the definition of $W$ in \eqref{sigma-weight-definition}. Now, since $W(z)^{-\left(2+\sigma\right)}\to\infty$ as $\left|z\right|\to\infty$, we have
			\begin{align}
				&\left|y-q_{n}'e_{2}\right|^{-(2+\sigma)}\sup_{z\in\mathbb{R}^{2}_{+}}W(z)^{-\left(2+\sigma\right)}\left|h(z)\right|\nonumber\\
				&\geq C_{0}\left|y-q_{n}'e_{2}\right|^{-(2+\sigma)}\left(\frac{\left(1+\left|y-q_{n}'e_{2}\right|\right)\left(1+\left|y+q_{n}'e_{2}\right|\right)}{2+\left|y-q_{n}'e_{2}\right|+\left|y+q_{n}'e_{2}\right|}\right)^{2+\sigma}\sup_{z\in\mathbb{R}^{2}_{+}}\left|h(z)\right|,\label{psi-n-goes-to-0-20}
			\end{align}
			some constant $C_{0}$. Now as $y\in\mathbb{R}^{2}_{+}$, $\left|y-q_{n}'e_{2}\right|\leq\left|y+q_{n}'e_{2}\right|$, and thus \eqref{psi-n-goes-to-0-20} becomes
			\begin{align*}
				\left|y-q_{n}'e_{2}\right|^{-(2+\sigma)}\sup_{z\in\mathbb{R}^{2}_{+}}W(z)^{-\left(2+\sigma\right)}\left|h(z)\right|&\geq C_{0}\left(1+\frac{1}{\left|y-q_{n}'e_{2}\right|}\right)^{2+\sigma}\left(\frac{1}{2}\right)^{2+\sigma}\sup_{z\in\mathbb{R}^{2}_{+}}\left|h(z)\right|\\
				&\geq C_{1} \sup_{z\in\mathbb{R}^{2}_{+}}\left|h(z)\right|,
			\end{align*}
			as $\left|y-q_{n}'e_{2}\right|\leq 3q_{n}'$. Thus we can say
			\begin{align}
				-\Delta\beta_{n}\geq C_{2}\left\|h_{n}\right\|_{L^{\infty}\left(\mathbb{R}_{+}^{2}\right)}\geq-\Delta\psi_{n},\nonumber
			\end{align}
			for some constant $C_{2}$ large enough, obtained by enlarging $c_{1}$ if necessary. We can also use \eqref{psi-n-K-n-boundary-bound} and \eqref{psi-n-h-n-annulus-relation} in the other direction to obtain
			\begin{align*}
				\beta_{n}(y)&\geq -\psi_{n}(y),\quad y\in \dell G_{n}\\
				-\Delta\beta_{n}&\geq\Delta\psi_{n},\quad y\in G_{n}.
			\end{align*}
			Thus by comparison with $\beta_{n}$ on $G_{n}$, we have
			\begin{align}
				\left\|\psi_{n}\right\|_{L^{\infty}\left(G_{n}\right)}\leq CK_{n}.\label{psi-n-goes-to-0-21}
			\end{align}
			Combining \eqref{psi-n-goes-to-0-21} with \eqref{psi-n-goes-to-0-15} gives us
			\begin{align}
				\left\|\psi_{n}\right\|_{L^{\infty}\left(\{\left|y-q_{n}'e_{2}\right|\geq2\rho\}\right)}\leq CK_{n}.\label{psi-n-goes-to-0-22}
			\end{align}
			Combining \eqref{psi-n-do-not-concentrate} and \eqref{psi-n-goes-to-0-22} gives, as $n\to\infty$, that
			\begin{align}
				\left\|\psi_{n}\right\|_{L^{\infty}\left(\mathbb{R}_{+}^{2}\right)}&\leq\max\left\{\left\|\psi_{n}\right\|_{L^{\infty}\left(B_{2\rho}(q'_{n}e_{2})\right)},\left\|\psi_{n}\right\|_{L^{\infty}\left(\{\left|y-q_{n}'e_{2}\right|\geq2\rho\}\right)}\right\}\nonumber\\
				&\leq\max\left\{\left\|\psi_{n}\right\|_{L^{\infty}\left(B_{2\rho}(q'_{n}e_{2})\right)}, C\left(\left\|h_{n}\right\|_{**}+\left\|\psi_{n}\right\|_{L^{\infty}\left(B_{\rho}(q'_{n}e_{2})\right)}+\left|a_{n}\right|+\left|b_{n}\right|\right)\right\}\to0,\label{psi-n-goes-to-0-23}
			\end{align}
			by the assumptions on $h_{n}$, Lemma \ref{estimate-on-coefficient-of-projection-lemma}, and \eqref{psi-n-do-not-concentrate}. We have that \eqref{psi-n-goes-to-0-23} gives us a contradiction on the $\psi_{n}$, as desired.
		\end{proof}
		We combine Lemmas \ref{estimate-on-coefficient-of-projection-lemma} and \ref{psi-a-priori-estimate-lemma} to obtain the main a priori estimate:
		\begin{lemma}\label{main-a-priori-estimate-lemma}
			Let $(q,c,\varepsilon)$ be chosen so that \eqref{rho-0-support}--\eqref{rho-0-definition} hold.. Suppose we have a solution to \eqref{projected-linear-problem-formulation-1}, with $h$ satisfying $\left\|h\right\|_{**}<\infty$ and the symmetry assumption \eqref{generalised-symmetry-assumptions}. Then
			\begin{align}
				\left\|\psi\right\|_{L^{\infty}\left(\mathbb{R}^{2}\right)}+\left(\left|a\right|+\left|b\right|\right)\leq C \left\|h\right\|_{**},\nonumber
			\end{align}
			some constant $C>0$.
		\end{lemma}
		We finish Section \ref{projected-linear-problem-subsection} by showing existence of solutions to the projected linear problem.
		\begin{theorem}
			Let $(q,c,\varepsilon)$ be chosen so that \eqref{rho-0-support}--\eqref{rho-0-definition} hold.. Then there is a linear operator $T_{q}$ from the space of $h$ satisfying \eqref{generalised-symmetry-assumptions} and $\left\|h\right\|_{**}<\infty$ to $Y$ such that
			$\psi\coloneqq T_{q}(h)$ is the unique solution to \eqref{projected-linear-problem-formulation-1} associated to that $h$. Moreover, we have the bound
			\begin{align}
				\left\|T_{q}(h)\right\|_{L^{\infty}\left(\mathbb{R}^{2}\right)}+\left(\left|a\right|+\left|b\right|\right)\leq C \left\|h\right\|_{**}.\label{existence-of-solution-to-linear-problem-statement}
			\end{align}
			for some $C$ independent of $\varepsilon$.
		\end{theorem}
		\begin{proof}
			We can reformulate the problem \eqref{projected-linear-problem-formulation-1} as
			\begin{align}
				\psi=K\psi+\Tilde{h},\nonumber
			\end{align}
			for an $\Tilde{h}$ depending linearly on $h$, and $K$ is a linear operator acting on $Y$ defined by
			\begin{align}
				(K\psi)(y)=\frac{1}{2\pi}\int_{\mathbb{R}_{+}^{2}}\log{\left(\frac{\left|z-\bar{y}\right|}{\left|z-y\right|}\right)}V(z)\psi(z)\ dz,\nonumber
			\end{align}
			and extending by symmetry to $\mathbb{R}^{2}$. Using methods analogous to those in the proof of Lemma \ref{psi-a-priori-estimate-lemma}, we obtain
			\begin{align}
				\left\|K\psi\right\|_{L^{\infty}\left(\mathbb{R}_{+}^{2}\right)}&\leq C\left|\log{\varepsilon}\right|\left\|\psi\right\|_{L^{\infty}\left(B_{\rho}\left(q'e_{2}\right)\right)}\nonumber\\
				\left\|\nabla K\psi\right\|_{L^{\infty}\left(\mathbb{R}_{+}^{2}\right)}&\leq C\left\|\psi\right\|_{L^{\infty}\left(B_{\rho}\left(q'e_{2}\right)\right)}\nonumber.
			\end{align}
			These bounds show compactness of $K$ for fixed $\varepsilon$, and by the Fredholm alternative, solutions to \eqref{projected-linear-problem-formulation-1} are unique for a particular $h$ as long as $\psi=K\psi$ only has the zero solution in $Y$, guaranteed by the a priori estimate stated in Lemma \ref{main-a-priori-estimate-lemma}. This lemma also gives the bound \eqref{existence-of-solution-to-linear-problem-statement}.
		\end{proof}
		\subsection{Projected Nonlinear Problem}
		We now use the linear theory developed in Section \ref{projected-linear-problem-subsection} to solve a projection of the full nonlinear problem. That is, we seek to construct a solution  to the problem
		\begin{align}
			L[\psi]&=E-N[\psi]+aVZ_{1}+bVZ_{2}\ \text{on}\ \mathbb{R}^{2},\nonumber\\
			\psi&\in Y,\label{projected-nonlinear-problem-formulation}
		\end{align}
		where we recall the definitions of $Y$ from \eqref{generalised-banach-space}, and of $E$ and $N[\psi]$ from \eqref{full-error-definition}--\eqref{full-nonlinearity-definition}.
		\begin{theorem}\label{projected-nonlinear-problem-theorem}
			Let $(q,c,\varepsilon)$ be chosen so that \eqref{rho-0-support}--\eqref{rho-0-definition} hold.. Then for $\varepsilon>0$ small enough, there is a unique solution $\psi_{q'}$ to \eqref{projected-nonlinear-problem-formulation} in the ball
			\begin{align*}
				\mathcal{B}\coloneqq\left\{\psi\in Y:\left\|\psi\right\|_{L^{\infty}\left(\mathbb{R}^{2}\right)}\leq \varepsilon\left|\log{\varepsilon}\right|\right\}.
			\end{align*}
			This solution has the bound
			\begin{align}
				\left\|\psi\right\|_{L^{\infty}\left(\mathbb{R}^{2}\right)}\leq C\varepsilon,\label{project-nonlinear-problem-bound}
			\end{align}
			where $C>0$ is a constant, and moreover the solution is continuous with respect to $q$.
		\end{theorem}
		
		\begin{proof}
			This is a fixed point argument. Indeed, we can define a map $\mathcal{A}$
			\begin{subequations}
				\begin{align}
					\mathcal{A}:Y&\rightarrow Y,\nonumber\\
					\psi&\mapsto T_{q}\left(-E+N[\psi]\right).\nonumber
				\end{align}
			\end{subequations}
			Then solving \eqref{projected-nonlinear-problem-formulation} is equivalent to solving
			\begin{align}
				\psi=\mathcal{A}(\psi).\nonumber
			\end{align}
			We run the fixed point procedure in the closed ball $\mathcal{B}\subset Y$ defined in Theorem \ref{projected-nonlinear-problem-theorem}.
			
			\medskip
			From \eqref{error-estimate}, \eqref{error-1-definition}, \eqref{error-2-defintion}, and \eqref{full-error-definition}, we can see that
			\begin{align}
				\left\|E\right\|_{L^{\infty}\left(\mathbb{R}^{2}\right)}\leq C_{1}\varepsilon.\nonumber
			\end{align}
			Moreover, for small enough $\varepsilon>0$, by \eqref{nonlinearity-1-definition}, \eqref{nonlinearity-2-definition}, and \eqref{full-nonlinearity-definition}, we have that $N[\psi]$ is supported in $B_{\rho}(q'e_{2})\cup B_{\rho}(-q'e_{2})$ for $\psi\in \mathcal{B}$, some $\rho>0$ independent of $\varepsilon$. This means the support of $N[\psi]$ can be contained inside a compact set uniformly for elements in this closed ball. Given the definition of $N[\psi]$ in \eqref{full-nonlinearity-definition}, we have the estimate
			\begin{align}
				\left\|N[\psi]\right\|_{L^{\infty}\left(\mathbb{R}^{2}\right)}\leq C_{2}\left\|\psi\right\|_{L^{\infty}\left(\mathbb{R}^{2}\right)}^{2}.\nonumber
			\end{align}
			Then we can show that if $\varepsilon$ is small enough, $\mathcal{A}:\mathcal{B}\rightarrow\mathcal{B}$ is a contraction, and admits a unique fixed point. Moreover, we can see that for this fixed point $\psi_{q'}$
			\begin{align}
				\left\|\psi_{q'}\right\|_{L^{\infty}\left(\mathbb{R}^{2}\right)}\leq C\left\|E\right\|_{L^{\infty}\left(\mathbb{R}^{2}\right)}\leq C\varepsilon,\label{error-bound-on-psi-q}
			\end{align}
			giving the bound \eqref{project-nonlinear-problem-bound}. Finally, the fixed point characterisation of the solution $\psi_{q'}$ gives continuity with respect to $q$.
		\end{proof}
		\subsection{Full Solution}
		We wish to solve 
		\begin{align}
			L[\psi_{q'}]&=E-N[\psi_{q'}]\ \text{on}\ \mathbb{R}^{2},\nonumber\\
			\psi_{q'}&\in X.\label{psi-q-fixed-point-problem-formulation-restatement}
		\end{align}
		Note that solving \eqref{projected-nonlinear-problem-formulation} with the condition that $\psi\in X\subset Y$ immediately forces $a=0$. So we are reduced to the problem of finding a relationship between $c$ and $q$ so that 
		\begin{align}
			L[\psi_{q'}]&=E-N[\psi_{q'}]+bVZ_{2},\nonumber\\
			\psi_{q'}&\in X,\label{full-solution-1}
		\end{align}
		has $b=0$.
		\begin{theorem}\label{c-q-relationship-lemma}
			Let $(q,c,\varepsilon)$ be chosen so that \eqref{rho-0-support}--\eqref{rho-0-definition} hold.. Then for $\varepsilon>0$ small enough, there is a unique $c=c(q)$ such that the corresponding solution of \eqref{full-solution-1} has $b=0$. This $c$ is of the form
			\begin{align}
				c=\frac{m}{2q}+g(q'),\label{c-q-relationship}
			\end{align}
			where $g(q')$ is continuous in $q'$.
		\end{theorem}
		\begin{proof}
			We know that our solution to \eqref{full-solution-1} $\psi_{q'}$ solves
			\begin{align}
				L\left[\psi_{q'}\right]=E-N\left[\psi_{q'}\right]+bVZ_{2}.\nonumber
			\end{align}
			Integrating against $Z_{2}$ gives
			\begin{align}
				b\underbrace{\left(\int_{\mathbb{R}^{2}}VZ_{2}^{2}\right)}_{\geq c_{0}}=\int_{\mathbb{R}^{2}}Z_{2}\left(L\left[\psi_{q'}\right]-E+N\left[\psi_{q'}\right]\right),\label{c-q-relationship-2}
			\end{align}
			where $c_{0}$ is some absolute positive constant. Thus, $b=0$ is equivalent to the right hand of \eqref{c-q-relationship-2} being $0$. We know from the proof of Lemma \ref{estimate-on-coefficient-of-projection-lemma}, and the bound on $\psi_{q'}$, \eqref{error-bound-on-psi-q}, that
			\begin{align}
				\left|\int_{\mathbb{R}^{2}}Z_{2}N\left[\psi_{q'}\right]\right|&\leq C\left\|\psi_{q'}\right\|_{L^{\infty}\left(\mathbb{R}^{2}\right)}^{2}\leq C\varepsilon^{2},\label{c-q-relationship-proof-nonlinearity-bound}\\
				\left|\int_{\mathbb{R}^{2}}Z_{2}L\left[\psi_{q'}\right]\right|&\leq C\varepsilon\left\|\psi_{q'}\right\|_{L^{\infty}\left(\mathbb{R}^{2}\right)}\leq C\varepsilon^{2}.\label{c-q-relationship-proof-linear-operator-bound}
			\end{align}
			We know from the definition of the error $E$ given in \eqref{error-1-definition}--\eqref{error-2-defintion}, and the definition of $Z_{2}$ given in \eqref{dq-approximate-kernel-element}, we have,
			\begin{align}
				\int_{\mathbb{R}^{2}}EZ_{2}&=\int_{B_{\rho}\left(q'e_{2}\right)}E_{1}Z_{2,1}+\int_{B_{\rho}\left(q'e_{2}\right)}E_{1}Z_{2,2}-\int_{B_{\rho}\left(-q'e_{2}\right)}E_{2}Z_{2,1}-\int_{B_{\rho}\left(-q'e_{2}\right)}E_{2}Z_{2,2}\nonumber\\
				&=2\left(\int_{B_{\rho}\left(q'e_{2}\right)}E_{1}Z_{2,1}+\int_{B_{\rho}\left(q'e_{2}\right)}E_{1}Z_{2,2}\right),\nonumber
			\end{align}
			with 
			\begin{align}
				\left|\int_{B_{\rho}\left(q'e_{2}\right)}E_{1}Z_{2,2}\right|\leq C\varepsilon^{2},\label{E1-Z22-first-bound}
			\end{align}
			and,
			\begin{align}
				&\int_{B_{\rho}\left(q'e_{2}\right)}E_{1}Z_{2,1}=-\int_{B_{\rho}\left(q'e_{2}\right)}\dell_{2}\left(\left(\Gamma(y-q'e_{2})\right)^{\gamma}_{+}\right)\left(\Gamma\left(2q'e_{2}\right)-\Gamma\left(y+q'e_{2}\right)+c\varepsilon\left(q'- y_{2}\right)\right)dy\nonumber\\
				&+\bigO{\left(\left|\Gamma\left(2q'e_{2}\right)-\Gamma\left(y+q'e_{2}\right)+c\varepsilon\left(q'- y_{2}\right)\right|^{2}\chi_{B_{\rho}\left(q'e_{2}\right)}\right)}\nonumber\\
				&=-\int_{B_{\rho}\left(q'e_{2}\right)}\dell_{2}\left(\left(\Gamma(y-q'e_{2})\right)^{\gamma}_{+}\right)\left(\Gamma\left(2q'e_{2}\right)-\Gamma\left(y+q'e_{2}\right)+c\varepsilon\left(q'- y_{2}\right)\right)dy+\bigO{\left(\varepsilon^{2}\right)}.\label{c-q-relationship-7}
			\end{align}
			Integrating by parts we have
			\begin{align}
				\int_{B_{\rho}\left(q'e_{2}\right)}E_{1}Z_{2,1}=&\int_{B_{\rho}\left(q'e_{2}\right)}\left(\Gamma(y-q'e_{2})\right)^{\gamma}_{+}\left(-\dell_{2}\Gamma\left(y+q'e_{2}\right)-c\varepsilon\right)dy+\bigO{\left(\varepsilon^{2}\right)}.\nonumber
			\end{align}
			Now we concentrate on calculating $-\dell_{2}\Gamma\left(y+q'e_{2}\right)-c\varepsilon$:
			\begin{align}
				-\dell_{2}\Gamma\left(y+q'e_{2}\right)-c\varepsilon=\frac{m\left(y_{2}+q'\right)}{\left|y+q'e_{2}\right|^{2}}-c\varepsilon=\varepsilon\left(\frac{m}{2q}-c\right)-\frac{m\varepsilon^{2}}{4q^{2}}\left(y_{2}-q'\right)+\bigO{\left(\varepsilon^{3}\right)},\label{c-q-relationship-7a}
			\end{align}
			where we have used the fact that we are on $B_{\rho}\left(q'e_{2}\right)$. We recall that
			\begin{align}
				\int_{B_{\rho}\left(q'e_{2}\right)}\left(\Gamma(y-q'e_{2})\right)^{\gamma}_{+}=\int_{B_{1}}\nu^{\gamma}_{+}=M,\label{c-q-relationship-10}
			\end{align}
			and using \eqref{c-q-relationship-proof-nonlinearity-bound}--\eqref{c-q-relationship-10}, we have
			\begin{align}
				b\left(\int_{\mathbb{R}^{2}}VZ_{2}^{2}\right)=2M\varepsilon\left(c-\frac{m}{2q}+\bigO{\left(\varepsilon\right)}\right),\label{c-q-relationship-11}
			\end{align}
			where $M$ is an absolute constant defined in \eqref{vortex-mass-definition}, and the remainder in the parentheses on the right hand side of order $\bigO{\left(\varepsilon\right)}$ is continuous in $q'$. Thus, there is a solution to $b=0$ of the form
			\begin{align}
				c=\frac{m}{2q}+g(q'),\label{c-q-relationship-12}
			\end{align}
			as required.
		\end{proof}
		We note that for $g$ defined in \eqref{c-q-relationship}, \eqref{c-q-relationship-11} and \eqref{c-q-relationship-12} imply that $g$ has the bound
		\begin{align}
			\left|g(q')\right|\leq C\varepsilon,\label{g-bound} 
		\end{align}
		for some absolute constant $C>0$. We can then use Theorem \ref{c-q-relationship-lemma} to improve the bounds on $E$, and therefore $\psi_{q'}$, and $g$ itself.
		\begin{lemma}\label{orthogonality-condition-error-improvement-lemma}
			Let $\left(q,c(q),\varepsilon\right)$ be as in Theorem \ref{c-q-relationship-lemma}. Then the corresponding solution to \eqref{psi-q-fixed-point-problem-formulation-restatement}, $\psi_{q'}$ has the bound
			\begin{align}
				\left\|\psi_{q'}\right\|_{L^{\infty}\left(\mathbb{R}^{2}\right)}&\leq C\varepsilon^{2},\label{orthogonality-condition-error-improvement-statement}\\
				\left|g(q')\right|&\leq C\varepsilon^{2}.\label{orthogonality-condition-error-improvement-statement-2}
			\end{align}
		\end{lemma}
		\begin{proof}
			As in \eqref{c-q-relationship-7}, one can see that
			\begin{align}
				E_{1}=-\gamma\left(\Gamma(y-q'e_{2})\right)^{\gamma-1}_{+}\left(\Gamma\left(2q'e_{2}\right)-\Gamma\left(y+q'e_{2}\right)+c\varepsilon\left(q'- y_{2}\right)\right)\nonumber\\
				+\bigO{\left(\left|\Gamma\left(2q'e_{2}\right)-\Gamma\left(y+q'e_{2}\right)+c\varepsilon\left(q'- y_{2}\right)\right|^{2}\chi_{B_{\rho}\left(q'e_{2}\right)}\right)}.\label{orthogonality-condition-error-improvement-1}
			\end{align}
			Expanding, we obtain that on the support of $E_{1}$
			\begin{align}
				&\Gamma\left(2q'e_{2}\right)-\Gamma\left(y+q'e_{2}\right)+c\varepsilon\left(q'- y_{2}\right)=-\varepsilon\left(y_{2}-q'\right)g(q')+\frac{m}{8\left(q'\right)^{2}}\left(y_{1}^{2}-\left(y_{2}-q'\right)^{2}\right)\nonumber\\
				&+\frac{m}{24\left(q'\right)^{3}}\left(4\left(y_{2}-q'\right)^{3}-3\left(y_{2}-q'\right)\left|y-q'e_{2}\right|^{2}\right)+\bigO{\left(\varepsilon^{4}\right)}=\bigO{\left(\varepsilon^{2}\right)},\label{orthogonality-condition-error-improvement-2}
			\end{align}
			by \eqref{g-bound}. Thus $E_{1}=\bigO{\left(\varepsilon^{2}\right)}$. The same bound holds for $E_{2}$, and therefore $E$. Consequently, we can improve \eqref{error-bound-on-psi-q} and get
			\begin{align}
				\left\|\psi_{q'}\right\|_{L^{\infty}\left(\mathbb{R}^{2}\right)}\leq C\left\|E\right\|_{L^{\infty}\left(\mathbb{R}^{2}\right)}\leq C\varepsilon^{2},\label{orthogonality-condition-error-improvement-3}
			\end{align}
			giving \eqref{orthogonality-condition-error-improvement-statement}. Then \eqref{orthogonality-condition-error-improvement-3} gives the bounds
			\begin{equation}\label{orthogonality-condition-error-improvement-4}
				\begin{aligned}
					\left|\int_{B_{\rho}\left(q'e_{2}\right)}E_{1}Z_{2,2}\right|&\leq C\varepsilon^{4},\\
					\left|\int_{\mathbb{R}^{2}}Z_{2}N\left[\psi_{q'}\right]\right|&\leq C\left\|\psi_{q'}\right\|_{L^{\infty}\left(\mathbb{R}^{2}\right)}^{2}\leq C\varepsilon^{4},\\
					\left|\int_{\mathbb{R}^{2}}Z_{2}L\left[\psi_{q'}\right]\right|&\leq C\varepsilon\left\|\psi_{q'}\right\|_{L^{\infty}\left(\mathbb{R}^{2}\right)}\leq C\varepsilon^{3},
				\end{aligned}
			\end{equation}
			improving the bounds \eqref{c-q-relationship-proof-nonlinearity-bound}--\eqref{E1-Z22-first-bound}. We note that the bound for the integral of $E_{1}Z_{2,2}$ comes from the fact that $Z_{2,2}$ is to main order a constant at order $\bigO{\left(\varepsilon\right)}$ with a remainder of order $\bigO{\left(\varepsilon^{2}\right)}$. Thus, due to the fact that we can represent $E_{1}$ using \eqref{orthogonality-condition-error-improvement-1}--\eqref{orthogonality-condition-error-improvement-2}, we can see that it is to main order a sum of a mode $1$ and mode $2$ term with respect to $y-q'e_{2}$, with remainder $\bigO{\left(\varepsilon^{3}\right)}$. Thus the product of the main order terms from $E_{1}$ and $Z_{2,2}$ integrate to give $0$ on $B_{\rho}\left(q'e_{2}\right)$. Next, \\eqref{orthogonality-condition-error-improvement-1}--\eqref{orthogonality-condition-error-improvement-2} also gives us
			\begin{align}
				E_{1}Z_{2,1}=-\dell_{2}\left(\left(\Gamma(y-q'e_{2})\right)^{\gamma}_{+}\right)\left(\Gamma\left(2q'e_{2}\right)-\Gamma\left(y+q'e_{2}\right)+c\varepsilon\left(q'- y_{2}\right)\right)+\bigO{\left(\varepsilon^{4}\right)}.\label{orthogonality-condition-error-improvement-6}
			\end{align}
			Finally, repeating the procedure of \eqref{c-q-relationship-7}--\eqref{c-q-relationship-12} with the improved bounds given in \eqref{orthogonality-condition-error-improvement-4}--\eqref{orthogonality-condition-error-improvement-6}, as well as now taking into account that the order $\varepsilon^{2}$ in the expansion \eqref{c-q-relationship-7a} is mode $1$, and is being integrated against a radial function to give us $0$, we conclude that
			\begin{align}
				0=2M\varepsilon\left(c-\frac{m}{2q}+\bigO{\left(\varepsilon^{2}\right)}\right),\nonumber
			\end{align}
			which gives \eqref{orthogonality-condition-error-improvement-statement-2}.
		\end{proof}
		\begin{remark}\label{mode-remark}
			Since $V$ is, to main order, radial, $\psi_{q'}$ exhibits an analogous modal expansion as in \eqref{orthogonality-condition-error-improvement-2} on a ball $B_{\rho}\left(q'e_{2}\right)$, some fixed $\rho>0$, with an analogous statement holding in coordinates $y+q'e_{2}$ on some ball $B_{\rho}(-q'e_{2})$.
		\end{remark}
		We now show how the vortex pair behaves upon $q'$-linearisation.
		\section{Properties of the Vortex Pair Under Linearisation}\label{q-linearisation-properties-section}
		Thus far we have found a family of solutions, $\Psi_{\varepsilon}$, to the problem
		\begin{align}
			\Delta\Psi_{\varepsilon}+f_{\varepsilon}(\Psi_{\varepsilon}-c\varepsilon y_{2})&=0,\nonumber\\
			\Psi_{\varepsilon}&\in X,\label{2d-euler-travelling-semilinear-elliptic-equation-epsilon-new}
		\end{align}
		for all $\varepsilon>0$ small enough, parameterised by $q$, where we recall the definition of $f_{\varepsilon}$ from \eqref{nonlinearity-definition}, and from \eqref{solution-banach-space} we recall that $X$ was the collection of $L^{\infty}$ functions on $\mathbb{R}^{2}$ even in $y_{1}$ and odd in $y_{2}$, as well as orthogonality condition \eqref{solution-Z2-orthogonality-condition}. We know this solution has the form
		\begin{align}
			\Psi_{\varepsilon}(y)=\Gamma\left(y-q'e_{2}\right)-\Gamma\left(y+q'e_{2}\right)+\psi_{q'}\left(y\right),\nonumber
		\end{align}
		where $c=m/2q + g(q')$, with $g(q')=\bigO{\left(\varepsilon^{2}\right)}$.
		
		\medskip
		We would like to show the existence of $\dell_{q'}\psi_{q'}$ and $\dell_{q'}c$, as well as proving bounds on these derivatives. To help with clarity when doing computations in this section, we introduce some notation. Recall that
		\begin{align}
			f_{\varepsilon}\left(\Psi_{\varepsilon}-c\varepsilon y_{2}\right)&=\left(\Gamma\left(y-q'e_{2}\right)-\Gamma\left(y+q'e_{2}\right)+\psi_{q'}-c\varepsilon y_{2}-\left|\log{\varepsilon}\right|\Omega\right)^{\gamma}_{+}\chi_{B_{\frac{s}{\varepsilon}}\left(q'e_{2}\right)}(y)\nonumber\\
			&-\left(\Gamma\left(y+q'e_{2}\right)-\Gamma\left(y-q'e_{2}\right)-\psi_{q'}+c\varepsilon y_{2}-\left|\log{\varepsilon}\right|\Omega\right)^{\gamma}_{+}\chi_{B_{\frac{s}{\varepsilon}}\left(-q'e_{2}\right)}(y).\nonumber
		\end{align}
		Thus
		\begin{align}
			f_{\varepsilon}'\left(\Psi_{\varepsilon}-c\varepsilon y_{2}\right)&=\gamma\left(\Gamma\left(y-q'e_{2}\right)-\Gamma\left(y+q'e_{2}\right)+\psi_{q'}-c\varepsilon y_{2}-\left|\log{\varepsilon}\right|\Omega\right)^{\gamma-1}_{+}\chi_{B_{\frac{s}{\varepsilon}}\left(q'e_{2}\right)}(y)\nonumber\\
			&+\gamma\left(\Gamma\left(y+q'e_{2}\right)-\Gamma\left(y-q'e_{2}\right)-\psi_{q'}+c\varepsilon y_{2}-\left|\log{\varepsilon}\right|\Omega\right)^{\gamma-1}_{+}\chi_{B_{\frac{s}{\varepsilon}}\left(-q'e_{2}\right)}(y).\nonumber
		\end{align}
		We let 
		\begin{equation}\label{derivative-of-nonlinearity-plus}
			\begin{aligned}
				\left(f_{\varepsilon}'\right)^{+}&=\gamma\left(\Gamma\left(y-q'e_{2}\right)-\Gamma\left(y+q'e_{2}\right)+\psi_{q'}-c\varepsilon y_{2}-\left|\log{\varepsilon}\right|\Omega\right)^{\gamma-1}_{+}\chi_{B_{\frac{s}{\varepsilon}}\left(q'e_{2}\right)}(y),\\
				\left(f_{\varepsilon}'\right)^{-}&=\gamma\left(\Gamma\left(y+q'e_{2}\right)-\Gamma\left(y-q'e_{2}\right)-\psi_{q'}+c\varepsilon y_{2}-\left|\log{\varepsilon}\right|\Omega\right)^{\gamma-1}_{+}\chi_{B_{\frac{s}{\varepsilon}}\left(-q'e_{2}\right)}(y).
			\end{aligned}
		\end{equation}
		Our first step is to see what equation the difference quotient of $\psi_{q'}$ solves. Note that varying $q$ does not change the symmetries in $y_{1}$ or $y_{2}$ for $\Psi_{\varepsilon}$, and so we can instead just work on $\mathbb{R}^{2}_{+}$, and use the fact that
		\begin{align}
			\Delta\psi_{q'}\left(y\right)&+\left(\Gamma\left(y-q'e_{2}\right)-\Gamma\left(y+q'e_{2}\right)+\psi_{q'}\left(y\right)-c(q)\varepsilon y_{2}-\left|\log{\varepsilon}\right|\Omega(q)\right)^{\gamma}_{+}\chi_{B_{\frac{s}{\varepsilon}}\left(q'e_{2}\right)}(y)\nonumber\\
			&=-\Delta\Gamma\left(y-q'e_{2}\right)\nonumber
		\end{align}
		on the upper half plane. Here we have used 
		\begin{align}
			\Delta\Gamma\left(y+q'e_{2}\right)=0,\ \  y_{2}\geq0,\nonumber
		\end{align}
		and also emphasised the dependence of $c$ and $\Omega$ on $q$. We define $\vartheta_{h}$, $h\in\mathbb{R}$, as the difference quotient for $\psi_{q'}$:
		\begin{align}
			\vartheta_{h}\left(y\right)\coloneqq\frac{\psi_{q'+h}\left(y\right)-\psi_{q'}\left(y\right)}{h}.\nonumber
		\end{align}
		Similarly, we define $\mathfrak{c}_{h}$ as
		\begin{align}
			\mathfrak{c}_{h}=\frac{c\left(q+\varepsilon h\right)-c\left(q\right)}{h}.\label{c-difference-quotient}
		\end{align}
		The nonlinearities for the semilinear elliptic equations $\psi_{q'}$ and $\psi_{q'+h}$ solve have different supports. However, we have that, for all $\varepsilon>0$ and $h\in\mathbb{R}$ small enough, these nonlinearities are both supported in
		\begin{align}
			B_{\rho'}\left(q'e_{2}\right)\cup B_{\rho'}\left(\left(q'+h\right)e_{2}\right)\subsetneq B_{\frac{s}{\varepsilon}}\left(q'e_{2}\right)\cap B_{\frac{s}{\varepsilon}}\left(\left(q'+h\right)e_{2}\right),\label{support-of-terms-in-intersection-of-balls}
		\end{align}
		some $\rho'>0$ independent of $\varepsilon$. In addition, we can say that the boundary of the set on the left hand side of \eqref{support-of-terms-in-intersection-of-balls} is positively separated from the boundary of the set on the right hand side. 
		For clarity, we also introduce the notation
		\begin{align}
			\mathscr{I}&\coloneqq B_{\frac{s}{\varepsilon}}\left(q'e_{2}\right)\cap B_{\frac{s}{\varepsilon}}\left(\left(q'+h\right)e_{2}\right).\nonumber
		\end{align}
		With \eqref{support-of-terms-in-intersection-of-balls} in mind, we define a quadratic term for $\left(\psi_{q'+h},\psi_{q'}\right)$:
		\begin{align}
			&\mathscr{N}\left[\psi_{q'+h},\psi_{q'}\right]\nonumber\\
			&=\frac{1}{h}\left(\Gamma\left(y-q'e_{2}\right)-\Gamma\left(y+q'e_{2}\right)+\psi_{q'+h}\left(y\right)-c\left(q\right)\varepsilon y_{2}-\left|\log{\varepsilon}\right|\Omega(q)\right)^{\gamma}_{+}\chi_{\mathscr{I}}\nonumber\\
			&-\frac{1}{h}\left(\Gamma\left(y-q'e_{2}\right)-\Gamma\left(y+q'e_{2}\right)+\psi_{q'}\left(y\right)-c\left(q\right)\varepsilon y_{2}-\left|\log{\varepsilon}\right|\Omega(q)\right)^{\gamma}_{+}\chi_{\mathscr{I}}\nonumber\\
			&-\gamma\vartheta_{h}(y)\left(\Gamma\left(y-q'e_{2}\right)-\Gamma\left(y+q'e_{2}\right)+\psi_{q'}\left(y\right)-c\left(q\right)\varepsilon y_{2}-\left|\log{\varepsilon}\right|\Omega(q)\right)^{\gamma-1}_{+}\chi_{\mathscr{I}}.\nonumber
		\end{align}
		We also define some error terms:
		\begin{align}
			&\mathscr{E}_{h}=\frac{1}{h}\left(\left(\Gamma\left(y-\left(q'+h\right)e_{2}\right)\right)^{\gamma}_{+}-\left(\Gamma\left(y-q'e_{2}\right)\right)^{\gamma}_{+}\right)\nonumber\\
			&-\frac{1}{h}\left(\Gamma\left(y-\left(q'+h\right)e_{2}\right)-\Gamma\left(y+\left(q'+h\right)e_{2}\right)+\psi_{q'+h}\left(y\right)-c\left(q+\varepsilon h\right)\varepsilon y_{2}-\left|\log{\varepsilon}\right|\Omega(q+\varepsilon h)\right)^{\gamma}_{+}\chi_{\mathscr{I}}\nonumber\\
			&+\frac{1}{h}\left(\Gamma\left(y-q'e_{2}\right)-\Gamma\left(y+q'e_{2}\right)+\psi_{q'+h}\left(y\right)-c\left(q+\varepsilon h\right)\varepsilon y_{2}-\left|\log{\varepsilon}\right|\Omega(q+\varepsilon h)\right)^{\gamma}_{+}\chi_{\mathscr{I}}.\nonumber
		\end{align}
		\begin{align}
			\mathscr{G}_{h}&=\frac{1}{h}\left(\Gamma\left(y-q'e_{2}\right)-\Gamma\left(y+q'e_{2}\right)+\psi_{q'+h}\left(y\right)-c\left(q+\varepsilon h\right)\varepsilon y_{2}-\left|\log{\varepsilon}\right|\Omega(q+\varepsilon h)\right)^{\gamma}_{+}\chi_{\mathscr{I}}\nonumber\\
			&-\frac{1}{h}\left(\Gamma\left(y-q'e_{2}\right)-\Gamma\left(y+q'e_{2}\right)+\psi_{q'+h}\left(y\right)-c\left(q\right)\varepsilon y_{2}-\left|\log{\varepsilon}\right|\Omega(q)\right)^{\gamma}_{+}\chi_{\mathscr{I}}.\nonumber
		\end{align}
		Then the equation $\vartheta_{h}$ solves on the upper half plane is
		\begin{align}
			\Delta\vartheta_{h}+\left(f_{\varepsilon}'\right)^{+}\vartheta_{h}=\mathscr{E}_{h}-\mathscr{G}_{h}-\mathscr{N}\left[\psi_{q'+h},\psi_{q'}\right],\label{difference-quotient-elliptic-equation}
		\end{align}
		where we recall the definition of $\left(f_{\varepsilon}'\right)^{+}$ from \eqref{derivative-of-nonlinearity-plus}. Note that due to \eqref{support-of-terms-in-intersection-of-balls}, all terms on the right hand side of \eqref{difference-quotient-elliptic-equation} are still differentiable.
		
		\medskip
		Now since we can explicitly differentiate $\Gamma\left(y\pm q'e_{2}\right)$ with respect to $q'$, and both $\psi_{q'}$ and $c(q)$ are continuous in $q'$, with uniform bounds with respect to $q$, we have that
		\begin{align}
			\left\|\mathscr{E}_{h}\right\|_{L^{\infty}\left(\mathbb{R}^{2}_{+}\right)}\leq C,\label{difference-quotient-error-estimate}
		\end{align}
		some $C$ independent of $h$, for all $h\in\mathbb{R}$ small enough.
		
		\medskip
		For the quadratic term, we have
		\begin{align}
			\left\|\mathscr{N}\left[\psi_{q'+h},\psi_{q'}\right]\right\|_{L^{\infty}\left(\mathbb{R}^{2}_{+}\right)}\leq C\left\|\vartheta_{h}\right\|_{L^{\infty}\left(\mathbb{R}^{2}_{+}\right)}\left\|\psi_{q'+h}-\psi_{q'}\right\|_{L^{\infty}\left(\mathbb{R}^{2}_{+}\right)}.\label{difference-quotient-quadratic-term-estimate}
		\end{align}
		Finally, for $\mathscr{G}_{h}$, we first calculate
		\begin{align}
			&\frac{\left(c\left(q+\varepsilon h\right)-c(q)\right)\varepsilon y_{2}+\left(\left|\log{\varepsilon}\right|\Omega\left(q+\varepsilon h\right)-\left|\log{\varepsilon}\right|\Omega(q)\right)}{h}\nonumber\\
			&=\varepsilon\mathfrak{c}_{h}\left(y_{2}-q'\right)+\left(\frac{m}{h}\log{\left(1+\frac{\varepsilon h}{q}\right)}-c\left(q+\varepsilon h\right)\varepsilon\right),\label{c-difference-quotient-error-estimate-1}
		\end{align}
		where we have used the definition of $\Omega$ in \eqref{Omega-definition-1}, and the definition of $\mathfrak{c}_{h}$ in \eqref{c-difference-quotient}. Thus
		\begin{align}
			\mathscr{G}_{h}=\mathscr{G}_{h}^{(1)}+\mathscr{G}_{h}^{(2)},\nonumber
		\end{align}
		where
		\begin{align}
			&\left|\mathscr{G}_{h}^{(1)}-\gamma\varepsilon\mathfrak{c}_{h}\left(y_{2}-q'\right)\left(\Gamma\left(y-q'e_{2}\right)-\Gamma\left(y+q'e_{2}\right)+\psi_{q'+h}\left(y\right)-c\left(q\right)\varepsilon y_{2}-\left|\log{\varepsilon}\right|\Omega(q)\right)^{\gamma-1}_{+}\chi_{\mathscr{I}}\right|\nonumber\\
			&\leq C\varepsilon,\label{c-difference-quotient-error-estimate-3}
		\end{align}
		and
		\begin{align}
			\left|\mathscr{G}_{h}^{(2)}\right|\leq C\varepsilon,\label{c-difference-quotient-error-estimate-4}
		\end{align}
		with $C$ independent of $h$. In both \eqref{c-difference-quotient-error-estimate-3} and \eqref{c-difference-quotient-error-estimate-4}, we have used the fact that we can bound
		\begin{align}
			\left|\left(c\left(q+\varepsilon h\right)-c(q)\right)\varepsilon y_{2}+\left(\left|\log{\varepsilon}\right|\Omega\left(q+\varepsilon h\right)-\left|\log{\varepsilon}\right|\Omega(q)\right)\right|\leq C\varepsilon,\label{c-difference-quotient-error-estimate-5}
		\end{align}
		for some $C$ independent of $h$, using the continuity of $c$ and $\Omega$ as functions of $q$.
		
		\medskip
		We also need to calculate the ``approximate orthogonality condition" that $\vartheta_{h}$ satisfies. Given both $\psi_{q'+h}$ and $\psi_{q'}$ satisfy \eqref{solution-Z2-orthogonality-condition}, we have
		\begin{align}
			\int_{\mathbb{R}^{2}_{+}}V\left(q'+h\right)Z_{2}\left(q'+h\right)\psi_{q'+h}-\int_{\mathbb{R}^{2}_{+}}V\left(q'\right)Z_{2}\left(q'\right)\psi_{q'}=0.\label{difference-quotient-approximate-orthogonality-condition-1}
		\end{align}
		We can rearrange this to
		\begin{align}
			\int_{\mathbb{R}^{2}_{+}}V\left(q'\right)Z_{2}\left(q'\right)\vartheta_{h}=\frac{1}{h}\int_{\mathbb{R}^{2}_{+}}\left(V\left(q'\right)Z_{2}\left(q'\right)-V\left(q'+h\right)Z_{2}\left(q'+h\right)\right)\psi_{q'+h}.\label{difference-quotient-approximate-orthogonality-condition-2}
		\end{align}
		Our first result is a bound on $\mathfrak{c}_{h}$.
		\begin{lemma}
			Let $\left(q,c(q),\varepsilon\right)$ be as in Theorem \ref{c-q-relationship-lemma}. Then for all $\varepsilon>0$ and $h\in\mathbb{R}$ small enough, we have the bound
			\begin{align}
				\left|\varepsilon\mathfrak{c}_{h}\right|\leq C\left(\varepsilon\left\|\vartheta_{h}\right\|_{L^{\infty}\left(\mathbb{R}^{2}_{+}\right)}+1\right),\label{c-difference-quotient-bound-statement}
			\end{align}
			some constant $C>0$ independent of $h$.
		\end{lemma}
		\begin{proof}
			The proof follows a similar strategy to the proof of Lemma \ref{estimate-on-coefficient-of-projection-lemma}, which can adapted from \cite{ao}. We integrate both sides of \eqref{difference-quotient-elliptic-equation} against $Z_{2}$ and obtain
			\begin{align}
				\int_{\mathbb{R}^{2}_{+}}\vartheta_{h}\left(\Delta Z_{2}+\left(f_{\varepsilon}'\right)^{+}Z_{2}\right)=\int_{\mathbb{R}^{2}_{+}}\mathscr{E}_{h}Z_{2}-\int_{\mathbb{R}^{2}_{+}}\mathscr{G}_{h}Z_{2}-\int_{\mathbb{R}^{2}_{+}}\mathscr{N}\left[\psi_{q'+h},\psi_{q'}\right]Z_{2}.\label{c-difference-quotient-bound-1}
			\end{align}
			Due to the bound on $\psi_{q'}$, \eqref{orthogonality-condition-error-improvement-statement}, we have
			\begin{align}
				\left|\int_{\mathbb{R}^{2}_{+}}\vartheta_{h}\left(\Delta Z_{2}+\left(f_{\varepsilon}'\right)^{+}Z_{2}\right)\right|\leq C\varepsilon\left\|\vartheta_{h}\right\|_{L^{\infty}\left(\mathbb{R}^{2}_{+}\right)}.\nonumber
			\end{align}
			Next,
			\begin{align}
				\left|\int_{\mathbb{R}^{2}_{+}}\mathscr{E}_{h}Z_{2}-\int_{\mathbb{R}^{2}_{+}}\mathscr{N}\left[\psi_{q'+h},\psi_{q'}\right]Z_{2}\right|\leq C\left(1+\varepsilon^{2}\left\|\vartheta_{h}\right\|_{L^{\infty}\left(\mathbb{R}^{2}_{+}\right)}\right),\nonumber
			\end{align}
			using \eqref{orthogonality-condition-error-improvement-statement} and \eqref{difference-quotient-quadratic-term-estimate}. Using \eqref{c-difference-quotient-error-estimate-1}--\eqref{c-difference-quotient-error-estimate-5}, we have
			\begin{align}
				\left|\int_{\mathbb{R}^{2}_{+}}\mathscr{G}_{h}Z_{2}-\mathcal{I}\right|\leq C\varepsilon,\nonumber
			\end{align}
			where $\mathcal{I}$ is given by
			\begin{align}
				\int_{\mathbb{R}^{2}_{+}}\gamma\varepsilon\mathfrak{c}_{h}\left(y_{2}-q'\right)\left(\Gamma\left(y-q'e_{2}\right)-\Gamma\left(y+q'e_{2}\right)+\psi_{q'+h}\left(y\right)-c\left(q\right)\varepsilon y_{2}-\left|\log{\varepsilon}\right|\Omega(q)\right)^{\gamma-1}_{+}\chi_{\mathscr{I}}Z_{2}.\nonumber
			\end{align}
			Thus it remains to bound $\mathcal{I}$. First note that 
			\begin{align}
				&\int_{\mathbb{R}^{2}_{+}}\gamma\left(y_{2}-q'\right)\left(\Gamma\left(y-q'e_{2}\right)-\Gamma\left(y+q'e_{2}\right)+\psi_{q'+h}\left(y\right)-c\left(q\right)\varepsilon y_{2}-\left|\log{\varepsilon}\right|\Omega(q)\right)^{\gamma-1}_{+}\chi_{\mathscr{I}}Z_{2}\nonumber\\
				&=\int_{\mathbb{R}^{2}_{+}}\frac{\left(y_{2}-q'\right)^{2}}{\left|y-q'e_{2}\right|}\nu'\left(\left|y-q'e_{2}\right|\right)\gamma\nu^{\gamma-1}\left(\left|y-q'e_{2}\right|\right)dy+\bigO{\left(\varepsilon\right)},\label{c-difference-quotient-bound-6}
			\end{align}
			where the $\bigO{\left(\varepsilon\right)}$ term on the right hand side of \eqref{c-difference-quotient-bound-6} has bound independent of $h$.
			
			\medskip
			Then we note that the first term on the right side of \eqref{c-difference-quotient-bound-6} is
			\begin{align}
				\int_{0}^{2\pi}\int_{0}^{1}\left(r^{2}\sin^{2}{\theta}\right)\nu'\left(r\right)\gamma\nu^{\gamma-1}\left(r\right)drd\theta=-c_{0},\label{c-difference-quotient-bound-7}
			\end{align}
			some absolute constant $c_{0}>0$. This is due to the fact that $\nu'<0$ for $r\in(0,1]$. Thus, combining \eqref{c-difference-quotient-bound-1}--\eqref{c-difference-quotient-bound-7} gives the desired result.
		\end{proof}
		Next we show that $\left\|\vartheta_{h}\right\|_{L^{\infty}\left(\mathbb{R}^{2}_{+}\right)}$ is uniformly bounded in $h$.
		\begin{lemma}
			Let $\left(q,c(q),\varepsilon\right)$ be as in Theorem \ref{c-q-relationship-lemma}. Then for all $\varepsilon>0$ small enough, there is an $h_{0}>0$ small enough such that we have the bound
			\begin{align}
				\sup_{\left|h\right|\leq h_{0}}\left\|\vartheta_{h}\right\|_{L^{\infty}\left(\mathbb{R}^{2}_{+}\right)}\leq C,\nonumber
			\end{align}
			some constant $C>0$.
		\end{lemma}
		\begin{proof}
			Suppose not, so that we can extract a sequence $h_{n}\to0$ such that $\left\|\vartheta_{h_{n}}\right\|_{L^{\infty}\left(\mathbb{R}^{2}_{+}\right)}\to\infty$ as $n\to\infty$. By the bound \eqref{c-difference-quotient-bound-statement}, and relabelling if necessary, we can assume
			\begin{align}
				\hat{\mathfrak{c}}\coloneqq\lim_{n\to\infty}\frac{\varepsilon\mathfrak{c}_{h_{n}}}{\left\|\vartheta_{h_{n}}\right\|_{L^{\infty}\left(\mathbb{R}^{2}_{+}\right)}}\nonumber
			\end{align}
			exists and has the bound
			\begin{align}
				\left|\hat{\mathfrak{c}}\right|\leq C\varepsilon.\label{difference-quotient-uniform-bound-2}
			\end{align}
			Next, \eqref{difference-quotient-error-estimate}--\eqref{c-difference-quotient-error-estimate-4} give that
			\begin{align}
				\frac{1}{\left\|\vartheta_{h_{n}}\right\|_{L^{\infty}\left(\mathbb{R}^{2}_{+}\right)}}&\left(\mathscr{E}_{h_{n}}-\mathscr{G}_{h_{n}}-\mathscr{N}\left[\vartheta_{h_{n}}\right]\right)\to \hat{\mathfrak{c}}\left(y_{2}-q'\right)\left(f_{\varepsilon}'\right)^{+},\ \ n\to\infty.\nonumber
			\end{align}
			Thus, after relabelling if necessary, we can assume
			\begin{align}
				\hat{\vartheta}_{h_{n}}\coloneqq \frac{\vartheta_{h_{n}}}{\left\|\vartheta_{h_{n}}\right\|_{L^{\infty}\left(\mathbb{R}^{2}_{+}\right)}}\to \hat{\vartheta},\nonumber
			\end{align}
			locally uniformly, some $\hat{\vartheta}$, as $n\to\infty$. Clearly $\|\hat{\vartheta}\|_{L^{\infty}\left(\mathbb{R}^{2}_{+}\right)}=1$. Moreover, it is even in $y_{1}$, and is $0$ on $\left\{y_{2}=0\right\}$. It solves the equation
			\begin{align}
				\Delta\hat{\vartheta}+\left(f_{\varepsilon}'\right)^{+}\hat{\vartheta}=\hat{\mathfrak{c}}\left(y_{2}-q'\right)\left(f_{\varepsilon}'\right)^{+}\nonumber
			\end{align}
			on the upper half plane, and from \eqref{difference-quotient-approximate-orthogonality-condition-1}--\eqref{difference-quotient-approximate-orthogonality-condition-2}, we get that
			\begin{align}
				\left|\int_{\mathbb{R}^{2}_{+}}VZ_{2}\hat{\vartheta}\right|\leq C\varepsilon^{3},\label{difference-quotient-uniform-bound-6}
			\end{align}
			using the bound on $\mathfrak{c}$ in \eqref{difference-quotient-uniform-bound-2}, and the bound on $\psi_{q'}$, \eqref{orthogonality-condition-error-improvement-statement}. Then, following the same strategy as the proof of Lemma \ref{psi-a-priori-estimate-lemma} with a minor adjustment for the fact we have the bound \eqref{difference-quotient-uniform-bound-6} instead of a true orthogonality condition, we obtain that $\hat{\vartheta}\equiv0$, a contradiction, as required.
		\end{proof}
		Now we show the existence of $\dell_{q'}\psi_{q'}$ and $\dell_{q'}c$.
		\begin{theorem}\label{existence-of-q-derivatives-lemma}
			Let $\left(q,c(q),\varepsilon\right)$ be as in Theorem \ref{c-q-relationship-lemma}, and be associated to the solution of \eqref{full-solution-1}, $\psi_{q'}$. Then for all $\varepsilon>0$ small enough, $\dell_{q'}\psi_{q'}$ and $\dell_{q'}c$ exist and have the bounds
			\begin{align}
				\left\|\dell_{q'}\psi_{q'}\right\|_{L^{\infty}\left(\mathbb{R}^{2}\right)}&\leq C,\label{existence-of-q-derivatives-statement-1}\\
				\left|\dell_{q'}c\right|&\leq C,\label{existence-of-q-derivatives-statement-2}
			\end{align}
			some constant $C>0$.
		\end{theorem}
		\begin{proof}
			We know that $\vartheta_{h}$ is bounded in $L^{\infty}$ uniformly in $h$. Moreover, using the fundamental solution to the Laplacian on $\mathbb{R}_{+}^{2}$ and that $\vartheta_{h}$ solves \eqref{difference-quotient-elliptic-equation}, we obtain that
			\begin{align}
				\left\|\nabla\vartheta_{h}\right\|_{L^{\infty}\left(\mathbb{R}^{2}_{+}\right)}\leq C,\nonumber
			\end{align}
			some $C>0$ independent of $h$. So take a sequence $h_{n}\to0$ as $n\to\infty$. We can extract a subsequence that converges locally uniformly to some $\vartheta_{1}\in L^{\infty}\left(\mathbb{R}^{2}\right)$ that is even in $y_{1}$ and $0$ on $\left\{y_{2}=0\right\}$. By relabelling if necessary, we have that $\varepsilon\mathfrak{c}_{h_{n}}$ also converges to some $\mathfrak{c}_{1}$ on this subsequence. Then $\left(\vartheta_{1},\mathfrak{c}_{1}\right)$ solve
			\begin{align}
				\Delta\vartheta_{1}+\left(f_{\varepsilon}'\right)^{+}\vartheta_{1}=\mathscr{E}+\mathfrak{c}_{1}\left(y_{2}-q'\right)\left(f_{\varepsilon}'\right)^{+},\label{existence-of-q-derivatives-2}
			\end{align}
			with $\mathscr{E}$ some error term. Now suppose that there is another subsequence that converges to $\left(\vartheta_{2},\mathfrak{c}_{2}\right)$. This pair then solve
			\begin{align}
				\Delta\vartheta_{2}+\left(f_{\varepsilon}'\right)^{+}\vartheta_{2}=\mathscr{E}+\mathfrak{c}_{2}\left(y_{2}-q'\right)\left(f_{\varepsilon}'\right)^{+}.\label{existence-of-q-derivatives-3}
			\end{align}
			Crucially, due to the fact that we can differentiate $\Gamma\left(y\pm q'e_{2}\right)$ with respect to $q'$, and because $\psi_{q'}$ and $c$ are continuous with respect to $q$, the error term $\mathscr{E}$ in both \eqref{existence-of-q-derivatives-2} and \eqref{existence-of-q-derivatives-3} are the same. Thus we have that
			\begin{align}
				\Delta\left(\vartheta_{1}-\vartheta_{2}\right)+\left(f_{\varepsilon}'\right)^{+}\left(\vartheta_{1}-\vartheta_{2}\right)=\left(\mathfrak{c}_{1}-\mathfrak{c}_{2}\right)\left(y_{2}-q'\right)\left(f_{\varepsilon}'\right)^{+}.\nonumber
			\end{align}
			Using \eqref{difference-quotient-approximate-orthogonality-condition-1}--\eqref{difference-quotient-approximate-orthogonality-condition-2}, we have
			\begin{align}
				\left|\int_{\mathbb{R}^{2}_{+}}VZ_{2}\left(\vartheta_{1}-\vartheta_{2}\right)\right|\leq C\left(\varepsilon\left\|\vartheta_{1}\right\|_{L^{\infty}\left(\mathbb{R}^{2}_{+}\right)}+\varepsilon\left\|\vartheta_{2}\right\|_{L^{\infty}\left(\mathbb{R}^{2}_{+}\right)}+1\right)\varepsilon^{2}.\label{existence-of-q-derivatives-5}
			\end{align}
			Then once again running the same strategy of proof as in Lemma \ref{psi-a-priori-estimate-lemma}, with the bound \eqref{existence-of-q-derivatives-5} instead of an orthogonality condition, we obtain that $\left(\vartheta_{1},\mathfrak{c}_{1}\right)=\left(\vartheta_{2},\mathfrak{c}_{2}\right)$, showing the existence of both $\dell_{q'}\psi_{q'}$ and $\dell_{q'}c$. Everything we did on the upper half plane can then be reflected to show existence on all of $\mathbb{R}^{2}$.
			
			\medskip
			Taking $h\to0$ in \eqref{difference-quotient-elliptic-equation}, and once again applying the techniques used in the proofs of Lemmas \ref{estimate-on-coefficient-of-projection-lemma} and \ref{psi-a-priori-estimate-lemma}, the bounds \eqref{difference-quotient-error-estimate}--\eqref{c-difference-quotient-error-estimate-4} imply the bounds \eqref{existence-of-q-derivatives-statement-1}--\eqref{existence-of-q-derivatives-statement-2}, as required.
		\end{proof}
		As in Lemma \ref{orthogonality-condition-error-improvement-lemma}, we now improve the bounds \eqref{existence-of-q-derivatives-statement-1}--\eqref{existence-of-q-derivatives-statement-2}.
		\begin{lemma}\label{bounds-on-q-derivatives-lemma}
			Let $\left(q,c(q),\varepsilon\right)$ be as in Theorem \ref{c-q-relationship-lemma}, and be associated to the solution of \eqref{full-solution-1}, $\psi_{q'}$. Then for all $\varepsilon>0$ small enough, $\dell_{q'}\psi_{q'}$ and $\dell_{q'}c$ have the bounds
			\begin{align}
				\left\|\dell_{q'}\psi_{q'}\right\|_{L^{\infty}\left(\mathbb{R}^{2}\right)}&\leq C\varepsilon^{2},\nonumber\\
				\left|\dell_{q'}c\right|&\leq C\varepsilon,\nonumber
			\end{align}
			some constant $C>0$. In addition, $g$ defined in Theorem \ref{c-q-relationship-lemma} has the bound
			\begin{align}
				\left|\dell_{q'}g\right|\leq C\varepsilon^{3},\nonumber
			\end{align}
			some constant $C>0$.
		\end{lemma}
		\begin{proof}
			We proceed in three steps.\\
			
			\noindent \textbf{Step I: Improvement of the bound on} $\left\|\dell_{q'}\psi_{q'}\right\|_{L^{\infty}\left(\mathbb{R}^{2}\right)}$.\\
			
			\noindent Differentiating \eqref{2d-euler-travelling-semilinear-elliptic-equation-epsilon-new} with respect to $q'$ and rearranging, we see that $\dell_{q'}\psi_{q'}$ solves
			\begin{align}
				\Delta\left(\dell_{q'}\psi_{q'}\right)+\left(f_{\varepsilon}'\right)^{+}\dell_{q'}\psi_{q'}=\mathscr{E}_{1}+\mathscr{E}_{2}\label{bounds-on-q-derivatives-1}
			\end{align}
			on $\mathbb{R}^{2}_{+}$, where
			\begin{align}
				\mathscr{E}_{1}&=\gamma\dell_{2}\Gamma\left(y-q'e_{2}\right)\left(\Gamma\left(y-q'e_{2}\right)-\Gamma\left(y+q'e_{2}\right)+\psi_{q'}\left(y\right)-c\varepsilon y_{2}-\left|\log{\varepsilon}\right|\Omega\right)^{\gamma-1}_{+}\chi_{B_{\frac{s}{\varepsilon}}\left(q'e_{2}\right)}\nonumber\\
				&-\gamma\dell_{2}\Gamma\left(y-q'e_{2}\right)\left(\Gamma\left(y-q'e_{2}\right)\right)^{\gamma-1}_{+},\nonumber
			\end{align}
			and
			\begin{align}
				\mathscr{E}_{2}=&\gamma\left(\dell_{2}\Gamma\left(y+q'e_{2}\right)+\dell_{q'}\left(c\right)\varepsilon y_{2}+\left|\log{\varepsilon}\right|\dell_{q'}\Omega\right)\cdot\nonumber\\
				&\cdot\left(\Gamma\left(y-q'e_{2}\right)-\Gamma\left(y+q'e_{2}\right)+\psi_{q'}\left(y\right)-c\varepsilon y_{2}-\left|\log{\varepsilon}\right|\Omega\right)^{\gamma-1}_{+}\chi_{B_{\frac{s}{\varepsilon}}\left(q'e_{2}\right)}.\nonumber
			\end{align}
			Differentiating the orthogonality condition \eqref{solution-Z2-orthogonality-condition} with respect to $q'$, and using Lemmas \ref{orthogonality-condition-error-improvement-lemma} and \ref{existence-of-q-derivatives-lemma} for bounds on $\psi_{q'}$ and $\dell_{q'}c$, we have
			\begin{align}
				\left|\int_{\mathbb{R}^{2}_{+}}VZ_{2}\dell_{q'}\psi_{q'}\right|\leq C\varepsilon^{2}.\label{bounds-on-q-derivatives-4}
			\end{align}
			Also due to Lemma \ref{orthogonality-condition-error-improvement-lemma}, we can see that
			\begin{align}
				\left\|\mathscr{E}_{1}\right\|_{L^{\infty}\left(\mathbb{R}^{2}_{+}\right)}\leq C\varepsilon^{2}.\nonumber
			\end{align}
			For $\mathscr{E}_{2}$, we have by \eqref{Omega-definition-1},
			\begin{align}
				\dell_{2}\Gamma\left(y+q'e_{2}\right)+\dell_{q'}\left(c\right)\varepsilon y_{2}+\left|\log{\varepsilon}\right|\dell_{q'}\Omega=\dell_{2}\Gamma\left(y+q'e_{2}\right)+\frac{m}{q'}+\dell_{q'}\left(c\right)\varepsilon y_{2}-\dell_{q'}\left(cq\right).\label{bounds-on-q-derivatives-7}
			\end{align}
			Also, using \eqref{c-q-relationship}, 
			\begin{align}
				\dell_{q'}\left(c\right)\varepsilon y_{2}-\dell_{q'}\left(cq\right)=\varepsilon\dell_{q'}\left(c\right)\left(y_{2}-q'\right)-\frac{m}{2q'}-\varepsilon g(q').\label{bounds-on-q-derivatives-8}
			\end{align}
			Putting \eqref{bounds-on-q-derivatives-7} and \eqref{bounds-on-q-derivatives-8} together, we have
			\begin{align}
				\dell_{2}\Gamma\left(y+q'e_{2}\right)+\dell_{q'}\left(c\right)\varepsilon y_{2}+\left|\log{\varepsilon}\right|\dell_{q'}\Omega=\left(\dell_{2}\Gamma\left(y+q'e_{2}\right)+\frac{m}{2q'}\right)+\varepsilon\dell_{q'}\left(c\right)\left(y_{2}-q'\right)-\varepsilon g(q').\label{bounds-on-q-derivatives-9}
			\end{align}
			We note that Lemma \ref{orthogonality-condition-error-improvement-lemma} implies $\varepsilon g(q')$ is $\bigO{\left(\varepsilon^{3}\right)}$ on $B_{\rho'}\left(q'e_{2}\right)$. The first term on the right hand side of \eqref{bounds-on-q-derivatives-9} is given by
			\begin{align}
				\dell_{2}\Gamma\left(y+q'e_{2}\right)+\frac{m}{2q'}&=-\frac{m\left(y_{2}+q'\right)}{\left|y+q'e_{2}\right|^{2}}+\frac{m}{2q'}\nonumber\\
				&=-\frac{m}{4\left(q'\right)^{2}}\left(\frac{y_{2}-q'}{1+\frac{y_{2}-q'}{q'}+\frac{\left|y-q'e_{2}\right|^{2}}{4\left(q'\right)^{2}}}+\frac{2q'}{1+\frac{y_{2}-q'}{q'}+\frac{\left|y-q'e_{2}\right|^{2}}{4\left(q'\right)^{2}}}\right)+\frac{m}{2q'}.\nonumber
			\end{align}
			Therefore
			\begin{align}
				&\left(\dell_{2}\Gamma\left(y+q'e_{2}\right)+\frac{m}{2q'}\right)+\varepsilon\dell_{q'}\left(c\right)\left(y_{2}-q'\right)=-\frac{m}{4\left(q'\right)^{2}}\left(y_{2}-q'\right)\left(1-\frac{y_{2}-q'}{q'}+\bigO{\left(\varepsilon^{2}\right)}\right)\nonumber\\
				&-\frac{m}{2q'}\left(1-\frac{\left(y_{2}-q'\right)}{q'}-\frac{\left|y-q'e_{2}\right|^{2}}{4\left(q'\right)^{2}}+\frac
				{\left(y_{2}-q'\right)^{2}}{\left(q'\right)^{2}}+\bigO{\left(\varepsilon^{3}\right)}\right)+\frac{m}{2q'}+\varepsilon\dell_{q'}\left(c\right)\left(y_{2}-q'\right).\nonumber
			\end{align}
			We can therefore say that
			\begin{align}
				&\dell_{2}\Gamma\left(y+q'e_{2}\right)+\dell_{q'}\left(c\right)\varepsilon y_{2}+\left|\log{\varepsilon}\right|\dell_{q'}\Omega=-\varepsilon g(q')\nonumber\\
				&+\left(-\frac{m}{4\left(q'\right)^{2}}+\frac{m}{2\left(q'\right)^{2}}+\varepsilon\dell_{q'}c\right)\left(y_{2}-q'\right)+\frac{m}{8\left(q'\right)^{3}}\left(y_{1}^{2}-\left(y_{2}-q'\right)^{2}\right)+\bigO{\left(\varepsilon^{4}\right)}.\label{bounds-on-q-derivatives-10c}
			\end{align}
			From the definition of $c$ in \eqref{c-q-relationship}, and Theorem \ref{existence-of-q-derivatives-lemma}, we in fact obtain
			\begin{align*}
				-\frac{m}{4\left(q'\right)^{2}}+\frac{m}{2\left(q'\right)^{2}}+\varepsilon\dell_{q'}c&=-\frac{m}{4\left(q'\right)^{2}}+\frac{m}{2\left(q'\right)^{2}}-\frac{m}{2\left(q'\right)^{2}}+\varepsilon\dell_{q'}g\\
				&=-\frac{m}{4\left(q'\right)^{2}}+\varepsilon\dell_{q'}g,
			\end{align*}
			so that \eqref{bounds-on-q-derivatives-10c} becomes
			\begin{align}
				&\dell_{2}\Gamma\left(y+q'e_{2}\right)+\dell_{q'}\left(c\right)\varepsilon y_{2}+\left|\log{\varepsilon}\right|\dell_{q'}\Omega=-\varepsilon g(q')+\left(-\frac{m}{4\left(q'\right)^{2}}+\varepsilon\dell_{q'}g\right)\left(y_{2}-q'\right)\nonumber\\
				&+\frac{m}{8\left(q'\right)^{3}}\left(y_{1}^{2}-\left(y_{2}-q'\right)^{2}\right)+\bigO{\left(\varepsilon^{4}\right)}.\label{bounds-on-q-derivatives-10d}
			\end{align}
			Thus, Lemma \ref{existence-of-q-derivatives-lemma}. which implies the boundedness of $\dell_{q'}g$, along with \eqref{bounds-on-q-derivatives-10d}, implies that
			\begin{align}
				\left\|\mathscr{E}_{2}\right\|_{L^{\infty}\left(\mathbb{R}^{2}_{+}\right)}\leq C\varepsilon.\label{bounds-on-q-derivatives-11}
			\end{align}
			Thus, once again adapting the methods of the proof of Lemma \ref{psi-a-priori-estimate-lemma}, using \eqref{bounds-on-q-derivatives-4} instead of a true orthogonality condition, we obtain
			\begin{align}
				\left\|\dell_{q'}\psi_{q'}\right\|_{L^{\infty}\left(\mathbb{R}^{2}\right)}\leq C\varepsilon,\label{bounds-on-q-derivatives-12}
			\end{align}
			where the bound over all of $\mathbb{R}^{2}$ once again comes from using the oddness in $y_{2}$.\\
			
			\noindent \textbf{Step II: Improvement of the bound on} $\dell_{q'}c$.\\
			
			\noindent We have the identity \eqref{c-q-relationship-2} with $b=0$, which we can then differentiate with respect to $q'$ to obtain
			\begin{align}
				0=&\int_{\mathbb{R}^{2}}\left(\dell_{q'}\psi_{q'}\right)L\left[Z_{2}\right]+\int_{\mathbb{R}^{2}}\psi_{q'}\dell_{q'}\left(L\left[Z_{2}\right]\right)-\dell_{q'}\left(\int_{\mathbb{R}^{2}}EZ_{2}\right)\nonumber\\
				&+\int_{\mathbb{R}^{2}}\left(\dell_{q'}Z_{2}\right)N\left[\psi_{q'}\right]+\int_{\mathbb{R}^{2}}Z_{2}\dell_{q'}\left(N\left[\psi_{q'}\right]\right).\label{bounds-on-q-derivatives-13}
			\end{align}
			Now by explicit calculation, we know that
			\begin{align}
				\left\|\dell_{q'}Z_{2}\right\|_{L^{\infty}\left(\mathbb{R}^{2}\right)}&\leq C,\label{bounds-on-q-derivatives-14}\\
				\left\|\dell_{q'}\left(L\left[Z_{2}\right]\right)\right\|_{L^{\infty}\left(\mathbb{R}^{2}\right)}&\leq C\varepsilon.\nonumber
			\end{align}
			For the quadratic term, on $\mathbb{R}^{2}_{+}$, we have
			\begin{align}
				\dell_{q'}\left(N\left[\psi_{q'}\right]\right)&=\dell_{q'}\left(\Gamma\left(y-q'e_{2}\right)-\Gamma\left(y+q'e_{2}\right)-c\varepsilon y_{2}-\left|\log{\varepsilon}\right|\Omega\right)\tilde{N}_{1}\left[\psi_{q'}\right]+\tilde{N}_{2}\left[\psi_{q'},\dell_{q'}\psi_{q'}\right],\nonumber
			\end{align}
			with
			\begin{align}
				\tilde{N}_{1}[\psi_{q'}]&=\gamma\left(\Gamma\left(y-q'e_{2}\right)-\Gamma\left(y+q'e_{2}\right)+\psi_{q'}\left(y\right)-c\varepsilon y_{2}-\left|\log{\varepsilon}\right|\Omega\right)^{\gamma-1}_{+}\chi_{B_{\frac{s}{\varepsilon}}\left(q'e_{2}\right)}\nonumber\\
				&-\gamma\left(\Gamma\left(y-q'e_{2}\right)-\Gamma\left(y+q'e_{2}\right)-c\varepsilon y_{2}-\left|\log{\varepsilon}\right|\Omega\right)^{\gamma-1}_{+}\chi_{B_{\frac{s}{\varepsilon}}\left(q'e_{2}\right)}\nonumber\\
				&-\gamma(\gamma-1)\left(\left(\Gamma\left(y-q'e_{2}\right)-\Gamma\left(y+q'e_{2}\right)-c\varepsilon y_{2}-\left|\log{\varepsilon}\right|\Omega\right)^{\gamma-2}_{+}\psi_{q'}(y)\right)\chi_{B_{\frac{s}{\varepsilon}}\left(q'e_{2}\right)},\nonumber
			\end{align}
			and
			\begin{align}
				&\tilde{N}_{2}\left[\psi_{q'},\dell_{q'}\psi_{q'}\right]\nonumber\\
				&=\gamma\left(\Gamma\left(y-q'e_{2}\right)-\Gamma\left(y+q'e_{2}\right)+\psi_{q'}\left(y\right)-c\varepsilon y_{2}-\left|\log{\varepsilon}\right|\Omega\right)^{\gamma-1}_{+}\chi_{B_{\frac{s}{\varepsilon}}\left(q'e_{2}\right)}\dell_{q'}\psi_{q'}(y)\nonumber\\
				&-\gamma\left(\Gamma\left(y-q'e_{2}\right)-\Gamma\left(y+q'e_{2}\right)-c\varepsilon y_{2}-\left|\log{\varepsilon}\right|\Omega\right)^{\gamma-1}_{+}\chi_{B_{\frac{s}{\varepsilon}}\left(q'e_{2}\right)}\dell_{q'}\psi_{q'}(y).\nonumber
			\end{align}
			Thus,
			\begin{align}
				\left\|\dell_{q'}\left(N\left[\psi_{q'}\right]\right)\right\|_{L^{\infty}\left(\mathbb{R}^{2}\right)}&\leq C\left(\left\|\psi_{q'}\right\|_{L^{\infty}\left(\mathbb{R}^{2}\right)}^{2}+\left\|\dell_{q'}\psi_{q'}\right\|_{L^{\infty}\left(\mathbb{R}^{2}\right)}\left\|\psi_{q'}\right\|_{L^{\infty}\left(\mathbb{R}^{2}\right)}\right)\nonumber\\
				&\leq C\varepsilon^{3},\label{bounds-on-q-derivatives-19}
			\end{align}
			using Lemma \ref{orthogonality-condition-error-improvement-lemma}, \eqref{bounds-on-q-derivatives-12}, and oddness in $y_{2}$.
			
			\medskip
			Thus \eqref{bounds-on-q-derivatives-14}--\eqref{bounds-on-q-derivatives-19} and Lemma \ref{orthogonality-condition-error-improvement-lemma} give
			\begin{align}
				\left|\int_{\mathbb{R}^{2}}\left[\left(\dell_{q'}\psi_{q'}\right)L\left[Z_{2}\right]+\psi_{q'}\dell_{q'}\left(L\left[Z_{2}\right]\right)+\left(\dell_{q'}Z_{2}\right)N\left[\psi_{q'}\right]+Z_{2}\dell_{q'}\left(N\left[\psi_{q'}\right]\right)\right]\right|\leq C\varepsilon^{2}.\label{bounds-on-q-derivatives-20}
			\end{align}
			Next we note
			\begin{align}
				\dell_{q'}\left(\int_{\mathbb{R}^{2}}EZ_{2}\right)=2\dell_{q'}\left(\int_{\mathbb{R}^{2}}E_{1}Z_{2,1}+\int_{\mathbb{R}^{2}}E_{1}Z_{2,2}\right),\nonumber
			\end{align}
			with $E_{1}$ and $Z_{2,i}$ being defined in \eqref{error-1-definition} and \eqref{dq-approximate-kernel-element}. Explicitly calculating the integral of $E_{1}Z_{2,2}$, differentiating with respect to $q'$ and applying the same modal considerations as in \eqref{orthogonality-condition-error-improvement-4} gives
			\begin{align}
				\left|\int_{\mathbb{R}^{2}}\dell_{q'}\left(E_{1}Z_{2,2}\right)\right|\leq C\varepsilon^{4}.\label{bounds-on-q-derivatives-22}
			\end{align}
			For $\dell_{q'}\left(E_{1}Z_{2,1}\right)$, we have
			\begin{align}
				&\dell_{q'}\left(E_{1}Z_{2,1}\right)\nonumber\\
				&=\dell_{q'}\left(\left(\Gamma(y-q'e_{2})\right)^{\gamma}_{+}-\left(\Gamma\left(y-q'e_{2}\right)-\Gamma\left(y+q'e_{2}\right)-c\varepsilon y_{2}-\left|\log{\varepsilon}\right|\Omega\right)^{\gamma}_{+}\chi_{B_{\frac{s}{\varepsilon}}\left(q'e_{2}\right)}\right)\dell_{2}\Gamma\left(y-q'e_{2}\right)\nonumber\\
				&+\left(\left(\Gamma(y-q'e_{2})\right)^{\gamma}_{+}-\left(\Gamma\left(y-q'e_{2}\right)-\Gamma\left(y+q'e_{2}\right)-c\varepsilon y_{2}-\left|\log{\varepsilon}\right|\Omega\right)^{\gamma}_{+}\chi_{B_{\frac{s}{\varepsilon}}\left(q'e_{2}\right)}\right)\left(-\dell_{2}^{2}\Gamma\left(y-q'e_{2}\right)\right).\nonumber
			\end{align}
			Then we have that
			\begin{align}
				\dell_{q'}\left(E_{1}Z_{2,1}\right)=\mathcal{E}_{1}+\mathcal{E}_{2}+\mathcal{E}_{3},\nonumber
			\end{align}
			where
			\begin{align}
				\mathcal{E}_{1}&=\gamma\dell_{2}\Gamma\left(y-q'e_{2}\right)\left(\dell_{2}\Gamma\left(y+q'e_{2}\right)+\dell_{q'}\left(c\right)\varepsilon y_{2}+\left|\log{\varepsilon}\right|\dell_{q'}\Omega\right)\cdot\nonumber\\
				&\cdot\left(\Gamma\left(y-q'e_{2}\right)-\Gamma\left(y+q'e_{2}\right)-c\varepsilon y_{2}-\left|\log{\varepsilon}\right|\Omega\right)^{\gamma-1}_{+}\chi_{B_{\frac{s}{\varepsilon}}\left(q'e_{2}\right)},\nonumber
			\end{align}
			\begin{align}
				\mathcal{E}_{2}&=-\gamma\left(\left(\Gamma(y-q'e_{2})\right)^{\gamma-1}_{+}-\left(\Gamma\left(y-q'e_{2}\right)-\Gamma\left(y+q'e_{2}\right)-c\varepsilon y_{2}-\left|\log{\varepsilon}\right|\Omega\right)^{\gamma-1}_{+}\chi_{B_{\frac{s}{\varepsilon}}\left(q'e_{2}\right)}\right)\cdot\nonumber\\
				&\cdot\left(\dell_{2}\Gamma\left(y-q'e_{2}\right)\right)^{2},\nonumber
			\end{align}
			\begin{align}
				\mathcal{E}_{3}&=-\left(\left(\Gamma(y-q'e_{2})\right)^{\gamma}_{+}-\left(\Gamma\left(y-q'e_{2}\right)-\Gamma\left(y+q'e_{2}\right)-c\varepsilon y_{2}-\left|\log{\varepsilon}\right|\Omega\right)^{\gamma}_{+}\chi_{B_{\frac{s}{\varepsilon}}\left(q'e_{2}\right)}\right)\cdot\nonumber\\
				&\cdot\dell_{2}^{2}\Gamma\left(y-q'e_{2}\right).\nonumber
			\end{align}
			Integrating $\mathcal{E}_{3}$, and using integration by parts, we get
			\begin{align}
				\int_{\mathbb{R}^{2}}\mathcal{E}_{3}=&\int_{\mathbb{R}^{2}}\gamma\dell_{2}\Gamma\left(y-q'e_{2}\right)\left(\dell_{2}\Gamma\left(y+q'e_{2}\right)+c\varepsilon\right)\cdot\nonumber\\
				&\cdot\left(\Gamma\left(y-q'e_{2}\right)-\Gamma\left(y+q'e_{2}\right)-c\varepsilon y_{2}-\left|\log{\varepsilon}\right|\Omega\right)^{\gamma-1}_{+}\chi_{B_{\frac{s}{\varepsilon}}\left(q'e_{2}\right)}\ dy-\int_{\mathbb{R}^{2}}\mathcal{E}_{2}.\nonumber
			\end{align}
			Thus
			\begin{align}
				\int_{\mathbb{R}^{2}}\left(\mathcal{E}_{1}+\mathcal{E}_{2}+\mathcal{E}_{3}\right)=&\int_{\mathbb{R}^{2}}\gamma\dell_{2}\Gamma\left(y-q'e_{2}\right)\left(2\dell_{2}\Gamma\left(y+q'e_{2}\right)+c\varepsilon+\dell_{q'}\left(c\right)\varepsilon y_{2}+\left|\log{\varepsilon}\right|\dell_{q'}\Omega\right)\cdot\nonumber\\
				&\cdot\left(\Gamma\left(y-q'e_{2}\right)-\Gamma\left(y+q'e_{2}\right)-c\varepsilon y_{2}-\left|\log{\varepsilon}\right|\Omega\right)^{\gamma-1}_{+}\chi_{B_{\frac{s}{\varepsilon}}\left(q'e_{2}\right)}\ dy.\nonumber
			\end{align}
			Now, as in Lemma \ref{orthogonality-condition-error-improvement-lemma}, we have
			\begin{align}
				&\gamma\left(\Gamma\left(y-q'e_{2}\right)-\Gamma\left(y+q'e_{2}\right)-c\varepsilon y_{2}-\left|\log{\varepsilon}\right|\Omega\right)^{\gamma-1}_{+}\chi_{B_{\frac{s}{\varepsilon}}\left(q'e_{2}\right)}\nonumber\\
				&=\gamma\left(\Gamma(y-q'e_{2})\right)^{\gamma-1}_{+}+\bigO{\left(\varepsilon^{2}\right)}.\label{bounds-on-q-derivatives-22a}
			\end{align}
			Moreover, similarly to \eqref{bounds-on-q-derivatives-7}--\eqref{bounds-on-q-derivatives-10c}, we have that
			\begin{align}
				2\dell_{2}\Gamma\left(y+q'e_{2}\right)+c\varepsilon+\dell_{q'}\left(c\right)\varepsilon y_{2}+\left|\log{\varepsilon}\right|\dell_{q'}\Omega=\varepsilon\dell_{q'}\left(g\right)\left(y_{2}-q'\right)+\bigO{\left(\varepsilon^{3}\right)}=\bigO{\left(\varepsilon\right)}.\nonumber
			\end{align}
			on the support of the integrand. Therefore,
			\begin{align}
				&\int_{\mathbb{R}^{2}}\left(\mathcal{E}_{1}+\mathcal{E}_{2}+\mathcal{E}_{3}\right)\nonumber\\
				&=\int_{\mathbb{R}^{2}}\dell_{2}\left(\left(\Gamma(y-q'e_{2})\right)^{\gamma}_{+}\right)\left(2\dell_{2}\Gamma\left(y+q'e_{2}\right)+c\varepsilon+\dell_{q'}\left(c\right)\varepsilon y_{2}+\left|\log{\varepsilon}\right|\dell_{q'}\Omega\right)\ dy+\bigO{\left(\varepsilon^{3}\right)}.\nonumber
			\end{align}
			Integrating by parts, we finally obtain
			\begin{align}
				\int_{\mathbb{R}^{2}}\dell_{q'}\left(E_{1}Z_{2,1}\right)=-&\int_{\mathbb{R}^{2}}\left(\Gamma(y-q'e_{2})\right)^{\gamma}_{+}\left(2\dell_{2}^{2}\Gamma\left(y+q'e_{2}\right)+\dell_{q'}\left(c\right)\varepsilon\right)\ dy+\bigO{\left(\varepsilon^{3}\right)}.\label{bounds-on-q-derivatives-33}
			\end{align}
			Recalling the explicit form of $\Gamma\left(y+q'e_{2}\right)$ on $B_{\rho}\left(q'e_{2}\right)$, we get
			\begin{align}
				2\dell_{2}^{2}\Gamma\left(y+q'e_{2}\right)+\dell_{q'}\left(c\right)\varepsilon=-\frac{m}{2\left(q'\right)^{2}}+\varepsilon\dell_{q'}c+\frac{2m}{\left(q'\right)^{3}}\left(y_{2}-q'\right)+\bigO{\left(\varepsilon^{4}\right)}.\label{bounds-on-q-derivatives-34}
			\end{align}
			Then, combining \eqref{bounds-on-q-derivatives-20}, \eqref{bounds-on-q-derivatives-22}, \eqref{bounds-on-q-derivatives-33}, and \eqref{bounds-on-q-derivatives-34}, the statement \eqref{bounds-on-q-derivatives-13} becomes
			\begin{align}
				2M\varepsilon\left(-\frac{1}{\varepsilon}\frac{m}{2\left(q'\right)^{2}}+\dell_{q'}c+\bigO{\left(\varepsilon\right)}\right)=0,\nonumber
			\end{align}
			where we recall $M$ is an absolute nonzero constant defined in \eqref{vortex-mass-definition}. Since we know that
			\begin{align}
				\frac{1}{\varepsilon}\frac{m}{2\left(q'\right)^{2}}=\frac{m\varepsilon}{2q^{2}}=\bigO{\left(\varepsilon\right)},\nonumber
			\end{align}
			we have that
			\begin{align}
				\left|\dell_{q'}c\right|\leq C\varepsilon,\label{bounds-on-q-derivatives-37}
			\end{align}
			with an analogous bound holding for $\dell_{q'}g$, for $g$ defined in \eqref{c-q-relationship}.\\
			
			\noindent \textbf{Step III: Repeating Steps I and II}.\\
			
			\noindent With the improved bound on $\dell_{q'}c$ in \eqref{bounds-on-q-derivatives-37}, we can now repeat the calculations \eqref{bounds-on-q-derivatives-1}--\eqref{bounds-on-q-derivatives-9} in Step I, and instead of obtaining the bound for $\mathscr{E}_{2}$ in \eqref{bounds-on-q-derivatives-11}, we get that
			\begin{align}
				\left\|\mathscr{E}_{2}\right\|_{L^{\infty}\left(\mathbb{R}^{2}_{+}\right)}\leq C\varepsilon^{2}.\nonumber
			\end{align}
			This implies that
			\begin{align}
				\left\|\dell_{q'}\psi_{q'}\right\|_{L^{\infty}\left(\mathbb{R}^{2}\right)}\leq C\varepsilon^{2},\nonumber
			\end{align}
			as claimed for $\dell_{q'}\psi_{q'}$. This improved bound then implies that 
			\begin{align}
				&\left|\int_{\mathbb{R}^{2}}\left[\left(\dell_{q'}\psi_{q'}\right)L\left[Z_{2}\right]+\psi_{q'}\dell_{q'}\left(L\left[Z_{2}\right]\right)+\left(\dell_{q'}Z_{2}\right)N\left[\psi_{q'}\right]+Z_{2}\dell_{q'}\left(N\left[\psi_{q'}\right]\right)\right]\right|\leq C\varepsilon^{4},\nonumber\\
				&\left|\int_{\mathbb{R}^{2}}\dell_{q'}\left(E_{1}Z_{2,2}\right)\right|\leq C\varepsilon^{4},\label{bounds-on-q-derivatives-38}
			\end{align}
			which lets us run Step II again to obtain that 
			\begin{align}
				2M\varepsilon\left(-\frac{1}{\varepsilon}\frac{m}{2\left(q'\right)^{2}}+\dell_{q'}c+\bigO{\left(\varepsilon^{2}\right)}\right)=0.\nonumber
			\end{align}
			Since differentiating \eqref{c-q-relationship} gives us
			\begin{align}
				\dell_{q'}c=-\frac{1}{\varepsilon}\frac{m}{2\left(q'\right)^{2}}+\dell_{q'}g(q'),\nonumber
			\end{align}
			we have
			\begin{align}
				\left|\dell_{q'}c\right|\leq C\varepsilon,\ \ \left|\dell_{q'}g\right|\leq C\varepsilon^{2}.\label{dell-q-g-epsilon-squared-bound}
			\end{align}
			\\
			\noindent \textbf{Step IV: Repeating Steps II and III}.\\
			
			\noindent With the improved bounds on $\dell_{q'}g$ in \eqref{dell-q-g-epsilon-squared-bound}, we can go back to \eqref{bounds-on-q-derivatives-22a}--\eqref{bounds-on-q-derivatives-33} and instead obtain  
			\begin{align*}
				&\gamma\left(\Gamma\left(y-q'e_{2}\right)-\Gamma\left(y+q'e_{2}\right)-c\varepsilon y_{2}-\left|\log{\varepsilon}\right|\Omega\right)^{\gamma-1}_{+}\chi_{B_{\frac{s}{\varepsilon}}\left(q'e_{2}\right)}\nonumber\\
				&=\gamma\left(\Gamma(y-q'e_{2})\right)^{\gamma-1}_{+}+\bigO{\left(\varepsilon^{2}\right)},
			\end{align*}
			\begin{align*}
				2\dell_{2}\Gamma\left(y+q'e_{2}\right)+c\varepsilon+\dell_{q'}\left(c\right)\varepsilon y_{2}+\left|\log{\varepsilon}\right|\dell_{q'}\Omega=\varepsilon\dell_{q'}\left(g\right)\left(y_{2}-q'\right)+\bigO{\left(\varepsilon^{3}\right)}=\bigO{\left(\varepsilon^{3}\right)},
			\end{align*}
			\begin{align*}
				\int_{\mathbb{R}^{2}}\dell_{q'}\left(E_{1}Z_{2,1}\right)=-&\int_{\mathbb{R}^{2}}\left(\Gamma(y-q'e_{2})\right)^{\gamma}_{+}\left(2\dell_{2}^{2}\Gamma\left(y+q'e_{2}\right)+\dell_{q'}\left(c\right)\varepsilon\right)\ dy+\bigO{\left(\varepsilon^{5}\right)}.
			\end{align*}
			We now also note that in the expansion \eqref{bounds-on-q-derivatives-34}, the order $\varepsilon^{3}$ is mode $1$, so when integrated against the radial function $\left(\Gamma(y-q'e_{2})\right)^{\gamma}_{+}$, will give $0$.\\
			
			Combining all this with the estimates in \eqref{bounds-on-q-derivatives-38}, the statement \eqref{bounds-on-q-derivatives-13} becomes 
			\begin{align*}
				2M\varepsilon\left(-\frac{1}{\varepsilon}\frac{m}{2\left(q'\right)^{2}}+\dell_{q'}c+\bigO{\left(\varepsilon^{3}\right)}\right)=0,
			\end{align*}
			which gives
			\begin{align*}
				\left|\dell_{q'}g\right|\leq C\varepsilon^{3},
			\end{align*}
			as required.
		\end{proof}
		The argument laid out in Section \ref{q-linearisation-properties-section} can be repeated multiple times for large enough $\gamma$. In particular for $\gamma>3$, we have
		\begin{align}
			\left\|\dell_{q'}^{2}\psi_{q'}\right\|_{L^{\infty}\left(\mathbb{R}^{2}\right)}+\left|\dell_{q'}^{2}c\right|+\varepsilon^{-2}\left|\dell_{q'}^{2}g\right|\leq C\varepsilon^{2}.\label{bounds-on-two-q-derivatives}
		\end{align}
		Then Theorem \ref{c-q-relationship-lemma}, Lemma \ref{orthogonality-condition-error-improvement-lemma}, Theorem \ref{existence-of-q-derivatives-lemma}, Lemma \ref{bounds-on-q-derivatives-lemma}, and \eqref{bounds-on-two-q-derivatives} give Theorem \ref{vortexpair-theorem}.
		
		\bigskip\noindent
		{\bf Acknowledgements:}

		J. Dávila has been supported by a Royal Society Wolfson Fellowship, UK.
		M. del Pino has been supported by the Royal Society Research Professorship grant RP-R1-180114 and by the ERC/UKRI Horizon Europe grant ASYMEVOL, EP/Z000394/1.
		M. Musso has been supported by EPSRC research grant EP/T008458/1.
		S. Parmeshwar has been supported by EPSRC research grants EP/T008458/1 and EP/V000586/1.
		
		We thank the referee for a careful reading of the manuscript and valuable suggestions.

		\bibliography{main} 
		\bibliographystyle{siam}
	\end{document}